\documentclass[12pt, twoside]{article}
\usepackage{amsmath,amsthm,amssymb}
\usepackage{times}
\usepackage{enumerate}

\usepackage{amsfonts}
\usepackage{amscd}
\usepackage[usenames, dvipsnames]{color}
\usepackage{endnotes}
\usepackage{epsfig}
\usepackage{eucal} 
\usepackage{graphpap}
\usepackage{latexsym}
\usepackage{pdfsync}
\usepackage{mathrsfs}

\def\titlerunning#1{\gdef\titrun{#1}}
\makeatletter
\def\author#1{\gdef\autrun{\def\and{\unskip, }#1}\gdef\@author{#1}}
\def\address#1{{\def\and{\\\hspace*{18pt}}\renewcommand{\thefootnote}{}
\footnote {#1}}
\markboth{\autrun}{\titrun}}
\makeatother
\def\email#1{e-mail: #1}
\def\subjclass#1{{\renewcommand{\thefootnote}{}
\footnote{\emph{Mathematics Subject Classification (2010):} #1}}}
\def\keywords#1{\par\medskip
\noindent\textbf{Keywords.} #1}

\newtheorem{thm}{Theorem}[section]
\newtheorem{cor}[thm]{Corollary}
\newtheorem{lem}[thm]{Lemma}
\newtheorem{prop}[thm]{Proposition}

\theoremstyle{definition}
\newtheorem{defin}[thm]{Definition}
\newtheorem{rem}[thm]{Remark}

\numberwithin{equation}{section}

\newcommand{\C}{\mathbb C}
\newcommand{\N}{\mathbb N}
\newcommand{\Q}{\mathbb Q}
\newcommand{\R}{\mathbb R}

\newcommand\bS{\mathbf S}

\newcommand{\cA}{\mathcal A}
\newcommand{\cB}{\mathcal B}
\newcommand{\cD}{\mathcal D}
\newcommand{\cE}{\mathcal E}

\newcommand{\cH}{\mathcal H}
\newcommand{\cI}{\mathcal I}
\newcommand{\cJ}{\mathcal J}
\newcommand{\cK}{\mathcal K}
\newcommand{\cM}{\mathcal M}
\newcommand{\cO}{\mathcal O}
\newcommand{\cS}{\mathcal S}

\newcommand{\cV}{\mathcal V}
\newcommand{\cW}{\mathcal W}

\newcommand\oA{\mathscr A}
\newcommand\oJ{\mathscr J}
\newcommand\oV{\mathscr V}
\newcommand\oW{\mathscr W}

\newcommand\e\varepsilon
\newcommand{\f}{\varphi}
\newcommand\bena{\mathbf 1}

\newcommand{\gota}{\text{\got a}}
\newcommand{\gotb}{\text{\got b}}
\newcommand{\g}{\text{\got g}}
\newcommand{\wh}\widehat

\newcommand{\pd}{\partial}
\newcommand{\leqsim}{\,\text{\posebni \char46}\,}
\newcommand{\geqsim}{\,\text{\posebni \char38}\,}

\newcommand{\nor}[1]{|\hskip -0.6pt | #1 |\hskip -0.6pt |}
\newcommand{\Nor}[1]{\left |\hskip -1.2pt\left | #1 \right |\hskip -1.2pt\right |}
\newcommand{\sk}[2]{\left\langle #1 , #2\right\rangle}
\newcommand{\mn}[2]{\{ #1\, ;\, #2 \}}
\newcommand{\Mn}[2]{\left\{ #1\, ;\, #2 \right\}}
\newcommand\norm[2]{{\left\Vert{#1}\right\Vert_{#2}}}

\renewcommand{\geq}{\geqslant}
\renewcommand{\leq}{\leqslant}
\renewcommand{\Re}{{\rm Re}\,}
\renewcommand{\Im}{{\rm Im}\,}
\renewcommand\th\vartheta
\renewcommand{\theta}\vartheta
\renewcommand{\div}{{\rm div}}
\renewcommand\mod[1]{\left\vert{#1}\right\vert}

\font\got=eufm10 at 12pt 
\font\posebni=msam10

\begin{document}

\baselineskip=17pt

\titlerunning{Bilinear embedding for divergence-form operators}

\title{Convexity of power functions and bilinear embedding for divergence-form operators with complex coefficients}

\date{}
\author{Andrea Carbonaro
\and 
Oliver Dragi\v{c}evi\'c}

\maketitle

\address{Carbonaro: Universit\`a degli Studi di Genova, Dipartimento di Matematica, Via Dodecaneso 35, 16146 Genova, Italy; \email{carbonaro@dima.unige.it}
\and
Dragi\v{c}evi\'c: Department of Mathematics, Faculty of Mathematics and Physics, University of Ljubljana, Jadranska 21, SI-1000 Ljubljana, Slovenia; \email{oliver.dragicevic@fmf.uni-lj.si}}


\begin{abstract}
We introduce a condition on accretive matrix functions, called $p$-ellipticity, and discuss its applications to the $L^p$ theory of elliptic PDE with complex coefficients. Our examples are: 
 {\it (i)} generalized convexity of power functions (Bellman functions),
 {\it (ii)} dimension-free bilinear embeddings,
 {\it (iii)} $L^p$-contractivity of semigroups, and
 {\it (iv)} holomorphic functional calculus. 
 Recent work by Dindo\v{s} and Pipher established close ties between $p$-ellipticity and
 {\it (v)} regularity theory of elliptic PDE with complex coefficients. 
The 
$p$-ellipticity condition 
arises from studying uniform positivity of a quadratic form associated with the matrix in question on one hand, and the Hessian of a power function on the other. Our results regarding contractivity extend earlier theorems by Cialdea and Maz'ya.

\end{abstract}

\section{Introduction and statement of the main results}
\label{s: introduction}

Suppose that $\Omega$ is an open subset of $\R^n$. Denote by $\cA(\Omega)$ the set of all complex {\it uniformly strictly accretive} (also called {\it elliptic}) $n\times n$ matrix functions on $\Omega$ with $L^{\infty}$ coefficients.

That is, the set of all measurable $A:\Omega\rightarrow\C^{n,n}$ for which 
\begin{itemize}
\item
there exists $\lambda>0$ such that for almost all $x\in\Omega$ we have
\begin{equation}
\label{eq: elli}
\Re\sk{A(x)\xi}{\xi}
\geq \lambda |\xi|^2\,,
\hskip 40pt \forall\xi\in\C^n;
\end{equation}
\item 
there exists $\Lambda>0$ such that for almost all $x\in\Omega$ we have
\begin{equation}
\label{eq: ne mogu se kontrolirati}
|\sk{A(x)\xi}{\eta}
|\leqslant \Lambda |\xi||\eta|\,,
\hskip 40pt \forall\xi,\eta\in\C^n.
\end{equation}
\end{itemize}
Elements of $\cA(\Omega)$ will also more simply be referred to as {\it accretive} or {\it elliptic matrices}.

Let $\cA_{\lambda,\Lambda}(\Omega)$ be the subset of $\cA(\Omega)$ in which \eqref{eq: elli} and \eqref{eq: ne mogu se kontrolirati} hold for fixed $\Lambda\geq\lambda>0$. Thus $\cA(\Omega)$ is the union of $\cA_{\lambda,\Lambda}(\Omega)$ over all $\Lambda\geq\lambda>0$. 
For any $A\in\cA(\Omega)$ denote by $\lambda_A$ the largest admissible $\lambda$ in \eqref{eq: elli} and by $\Lambda_A$ the smallest $\Lambda$ in \eqref{eq: ne mogu se kontrolirati}.

Define formally the operator $L=L_{A}$ by
$
Lu=-\div(A\nabla u).
$ 
A standard way of interpreting $Lu$ is via sesquilinear forms; 
we follow \cite{O} here. 
Define $\cD(\text{\got a}):=H_0^1(\Omega)$, the closure of $C_c^\infty(\Omega)$ in the Sobolev space $H^1(\Omega)=W^{1,2}(\Omega)$, and
\begin{equation}
\label{eq: miruna}
\text{\got a}(u,v):=\int_{\Omega}\sk{A\nabla u}{\nabla v}_{\C^n}
\hskip 40pt
\text{for } u,v\in\cD(\text{\got a}).
\end{equation}
Recall that in the special case $\Omega=\R^n$ we have $H^1(\Omega)=H_0^1(\Omega)$, e.g. 
\cite[Corollary~3.19]{Adams}.

The associated operator $L$ is defined by the requirement that
\begin{equation*}
\label{eq: prokofjev no.5}
\text{\got a}(u,v)=\sk{L_{A}u}{v}_{L^2(\Omega)},\quad u\in\cD(L_{A}),\quad v\in\cD(\text{\got a})\,,
\end{equation*} 
where
 $$
 \cD(L_{A})=\mn{u\in\cD(\text{\got a})}
 {\exists w\in L^2(\Omega):\ 
 \text{\got a}(u,v)=\sk{w}{v}_{L^2(\Omega)}\ \forall v\in 
 \cD(\text{\got a})}
 $$
and $L_{A}u:=w$. 
So $L_{A}$ equals $-\div(A\nabla)$ subject to the Dirichlet boundary conditions on $\Omega$.

The form $\text{\got a}$ is densely defined, accretive, continuous, closed and sectorial \cite[pp. 99-101]{O}. Therefore \cite[Proposition~1.27 and Theorem~1.54]{O}, the associated operator $-L_{A}$ generates on $L^2(\Omega)$ a strongly continuous semigroup of operators
$$
P^{A}_{t}=\exp(-tL_{A}),\quad t>0,
$$
which is analytic and contractive in a cone of positive angle. Hence $P^{A}_{t}$ maps  $L^{2}(\Omega)$ into $\cD(L_{A})\subseteq H^{1}_{0}(\Omega)$ \cite[Theorem II.4.6]{EN}, thus the spatial gradient $\nabla  P_{t}^{A}f$ is always well defined. By \cite[p. 72]{Stein red}, given $f\in L^{2}(\Omega)$ we can redefine each $P^{A}_{t}f$ on a set of measure zero, in such a manner that for almost every $x\in\Omega$ the function $t\mapsto P^{A}_{t}f(x)$ is real-analytic on $(0,\infty)$.

\medskip
The $L^2$ theory of the operators $L_{A}$ with $A\in\cA(\R^n)$ has been developed since the pioneering works in the 50's and 60's, see \cite{K, O, McI} for references. It is well known that they admit a unique {\it square root} $L_{A}^{1/2}$, see e.g. \cite{McI}. The last major piece of the $L^2$ theory was completed in 2002 by Auscher et al. \cite{AHLMcT} who proved the long-standing {\it Kato conjecture} which asserted that $\cD\big(L_{A}^{1/2}\big)=\cD(\text{\got a})$. An important part of the proof were specific square function estimates. 

As for the $L^{p}$ theory for $p\not=2$, there have been many results prior to and following the resolution of the Kato's conjecture. See \cite{A} and \cite{AT} for a comprehensive survey.

\medskip
For $p>1$ define $q=p/(p-1)$ and $p^*=\max\{p,q\}$. 
\label{what did they call him}
In \cite{DV-Kato}, A. Volberg and the second author of the present paper proved the following ``dimension-free'' {\it bilinear embedding}.

\setcounter{thm}{-1}
\begin{thm}
\label{t: oktyabr'}
Let $A\in\cA(\R^n)$ be a real accretive matrix function and $L=L_{A}$ the associated divergence-form operator, as above.
For any $p>1$ and any $f,g\in C_c^\infty(\R^n)$, 
\begin{equation*}
\int_0^\infty\int_{\R^n} |\nabla  P^{A}_{t}f(x)|\,|\nabla  P^{A}_{t}g(x)|\,\,dx\,dt\,\leqsim
\,p^*\,\nor{f}_p\nor{g}_q\,.
\end{equation*}
\end{thm}

Dealing with divergence-form operators associated to {\sl complex} matrices is known to be considerably more difficult (see, for example, \cite{tERS, AMcT1, AT} or \cite[Section~8.6]{Arendt} for more illustrations of this fact), and this is the subject of investigation in the present paper. We introduce a new condition, 
called $p$-ellipticity, and argue that it may be of interest for the $L^{p}$ theory of elliptic partial differential operators with complex coefficients.
In particular, it sheds more light onto the $L^p$ estimates for the heat semigroup with complex coefficients; see Sections \ref{s: saraj} and \ref{s: zavratnica} for historical information on this topic and the connection with the bilinear embedding.

\subsection{New findings (summary)}

Here we give a brief overview of the results in this paper. A detailed account, including exact formulations and the historical background, is postponed to the remaining parts of Section \ref{s: introduction}. We direct the reader to these parts at the appropriate spots in the following paragraphs.

\medskip
Forgetting for the moment about the control over the constants, Theorem~\ref{t: oktyabr'} follows from the $L^{p}$ boundedness of the conical square function, which is due to Auscher, Hofmann and Martell \cite{AHM}. They also considered complex $A$ and showed that square function estimates extend in the range $p\in(p_{-}(L_{A}),\infty)$, with the lower bound being sharp \cite[Theorem~3.1.(2)]{AHM}. Here $(p_{-}(L_{A}),p_{+}(L_{A}))$ is the maximal open interval of indices $p$ for which $\left(P_{t}^A\right)_{t>0}$ is uniformly bounded on $L^{p}$. As a consequence, their estimates imply a (``non-dimension-free") bilinear embedding in the range $(p_-(L_{A}),p_+(L^{*}_{A}))$, see \cite[p. 5471]{AHM}. 
More information on the above square function estimates and their relation to the bilinear embedding can be found in Section \ref{s: mota}.

We are interested in extending both Theorem~\ref{t: oktyabr'} and the above corollary of Auscher, Hofmann and Martell, meaning that we aim for a {\it dimension-free} bilinear embedding for {\it complex} accretive matrices. Since Auscher et al. showed that their result is closely related to {\sl boundedness} of $(P^{A}_{t})_{t>0}$, it is reasonable to expect that our dimension-free extension should be linked to the {\sl contractivity} of $(P^{A}_{t})_{t>0}$, see also \cite{CD-mult, CD-OU}. Indeed, we confirm this thought by finding a condition \eqref{eq: aferim jano} which is linked both to bilinear embeddings (Theorem~\ref{t: bilincomplex}) and contractivity (Theorem \ref{t: Aufnahmebetriebsartenwaehler}).

The first result of ours is an extension of Theorem~\ref{t: oktyabr'} to the case of complex matrices. It is summarized in Theorem~\ref{t: bilincomplex}. Unlike in the real case, in the complex one the (dimension-free) embedding does not necessarily hold for all $p\in(1,\infty)$ but only in a specific interval; see Section \ref{s: vedno znova lotu}. In Section~\ref{s: pilot} we introduce a new condition \eqref{eq: aferim jano} which we call {\it $p$-ellipticity} and 
which determines the endpoints of that interval. As described later, this condition turns out to be closely related to several phenomena in analysis and PDE which may occur in the presence of complex accretive matrices. Moreover, motivated by a question by P. Auscher, we consider in Theorem \ref{t: bilincomplex} a more general type of embedding, namely such obtained by applying {\sl different} semigroups to $f$ and $g$. 
The estimates we obtain are explicit and involve $p,\lambda,\Lambda$, but do not depend on the dimension $n$, thus retaining the ``dimension-free'' nature of Theorem~\ref{t: oktyabr'}. 
They are discussed in Sections \ref{s: bilincomplex} and \ref{s: vedno znova lotu}. 

In order to prove Theorem~\ref{t: bilincomplex} we adapt the so-called Bellman-function-heat-flow method. It is described in Section~\ref{s: sokolov hortus musicus}. We focus on analyzing properties of one particular Bellman function \eqref{eq: rupkina} and its principal building blocks -- power functions.

The sharp condition \eqref{eq: aferim jano} is equivalent to the generalized convexity of power functions $F_p$, that is, to the (uniform) positivity of certain quadratic forms associated with $A$ and ${\rm Hess}\,F_p$; see Section~\ref{ah moj aljo} and Proposition~\ref{p: kaminszki}. 
Previously, the importance of this positivity was recognized and studied in a few special cases: when $A$ is either the identity \cite{NT, DV-Sch}, real accretive 
\cite{DV-Kato}, of the form $e^{i\phi}I$ \cite{CD-mult}, or of the form $e^{i\phi}B$ with $B$ real, constant and with a symmetric part which is positive definite \cite{CD-OU}. 
Such problems are also related to similar questions considered earlier by Bakry; see \cite[Th\'eor\`eme 6]{Ba}. The  current paper brings a systematic approach to convexity of power functions in presence of arbitrary uniformly strictly accretive complex matrix functions $A$, see Section~\ref{s: garanča}.

It has been known in some cases, see \cite{CD-mult, CD-OU}, that dimension-free bilinear embeddings may be related to the contractivity of the associated semigroups on $L^{p}$. In the context of divergence-form operators, recent results regarding $L^{p}$-contractivity of $(P^{A}_{t})_{t>0}$ have been obtained by Cialdea and Maz'ya \cite{CM, CM2014}. While the exact range of this contractivity is still not known for arbitrary $A\in\cA(\Omega)$, we manage to narrow the gap between sufficient and necessary conditions by improving some of the results from \cite{CM} and \cite{O}. 
The reader will find detailed information on this in Section \ref{s: saraj}. 
In particular, we characterize the $L^{p}$-contractivity of $(P^{A}_{t})_{t>0}$ when the distributional divergence of every column of the antisymmetric part of the imaginary part of $A$ is zero, see Theorem~\ref{t: Aufnahmebetriebsartenwaehler}. 
Approaching contractivity via the classical Lumer--Phillips theorem proved to be difficult due to domain issues. Instead, we sought a characterization in terms of sesquilinear forms and found a convenient result by Nittka \cite[Theorem~4.1]{N} on which we then relied in proving Theorem~\ref{t: Aufnahmebetriebsartenwaehler}. See also page \pageref{sensordust} for connections with ter Elst et al. \cite{tELSV}.

Another area where $p$-ellipticity \eqref{eq: aferim jano} is a central condition is holomorphic functional calculus in sectors. The present authors obtained sharp results for generators of symmetric contraction semigroups \cite{CD-mult} and for nonsymmetric Ornstein-Uhlenbeck operators \cite{CD-OU}. It emerged that the opening angles of the optimal sectors are naturally interlaced with $p$-ellipticity. See Section~\ref{s: Howlin Wolf Hidden Charms} and Remark \ref{r: stojkovic} for explanation.

The key condition \eqref{eq: aferim jano} also bears deep connections with the regularity theory of
elliptic PDE. This was recently discovered by Dindo\v s and Pipher \cite{DiPi} while developing their program of studying solutions to the divergence-form operators with complex coefficients and the associated boundary value problems. They found the sharp condition which permits proving reverse H\"older inequalities for weak solutions of $L_{A}$ with complex $A$. 
It turns out that this condition is precisely a reformulation of $p$-ellipticity \eqref{eq: aferim jano}.
These inequalities serve as a replacement for the De Giorgi--Nash--Moser regularity theory for real $A$. As an application, they solve $L^{p}$ Dirichlet problems for $L_{A}$ in the range of $p$ determined by $p$-ellipticity.

To summarize, the condition we introduce in this paper, that is, the $p$-ellipticity \eqref{eq: aferim jano},
lies at the junction of several different directions in analysis and PDE: 

\begin{enumerate}[i)]

\item convexity of power functions (Bellman functions), 

\item dimension-free bilinear embeddings, 

\item $L^p$-contractivity of semigroups $(P^{A}_{t})_{t>0}$, 

\item holomorphic functional calculus, and 

\item regularity theory of elliptic PDE with complex coefficients (Dindo\v s and Pipher \cite{DiPi}).

\end{enumerate}

The rest of this section is devoted to giving precise formulation of the results announced above and the motivation which led to our pursuing them.

\subsection{The new condition}
\label{s: pilot}

For $p\in[1,\infty]$ define the $\R$-linear map $\cI_p:\C^n\rightarrow\C^n$ by
\begin{equation}
\label{eq: nindze za 30kn}
\cI_p\xi=\xi+(1-2/p)\bar\xi.
\end{equation}
Equivalently, with $q\in[1,\infty]$ given by $1/p+1/q=1$,
\begin{equation}
\label{eq: krokodil slavko}
\cI_p(\alpha+i\beta)=
2\left(\frac\alpha q+i\,\frac\beta p\right)
\hskip 50pt
\forall\,\alpha,\beta\in\R^n.
\end{equation}
For any open set $\Omega\subset\R^n$ and a bounded matrix function $A:\Omega\rightarrow\C^{n,n}$, define\footnote{In early versions of this paper we used in \eqref{eq: kabuto}
the operator $\cJ_p=\cI_q/2$ instead of $\cI_p$ and a normalization of $2$ in front of the ess\,inf. In both cases the quantity represented by $\Delta_p(A)$ is the same.}
\begin{equation}
\label{eq: kabuto}
\Delta_p(A):=
\underset{x\in\Omega}{{\rm ess}\inf}
\min_{\substack{\xi\in\C^{n}\\ |\xi|=1}}
\Re\sk{A(x)\xi}{\cI_p\xi}_{\C^n}.
\end{equation}
  We say that $A$ is {\it $p$-elliptic} if
  \begin{equation}
  \label{eq: aferim jano}
  \Delta_p(A)>0,
  \end{equation}
  that is, if  there exists $C=C(p,A)>0$ such that for almost every $x\in\Omega$ we have
  \begin{equation}
  \label{eq: Till I collapse}
  \Re\sk{A(x)\xi}{\cI_p\xi}_{\C^n}
  \geq C|\xi|^2\,,
  \hskip 20pt \forall\xi\in\C^n.
  \end{equation}
By the last part of Proposition \ref{p: kaminszki} we get a straightforward characterization of $p$-ellipticity which inherently involves the invariance of this condition under conjugation of $p$: 
  \begin{equation*}
a.e.\ x\in\Omega:
\hskip 40pt
  \Re\sk{A(x)\xi}{
  \xi+|1-2/p|\bar\xi}_{\C^n}
  \geqsim |\xi|^2\,,
  \hskip 20pt \forall\xi\in\C^n.
  \end{equation*}
  So $p$-ellipticity is just ``one benign operator ($\cI_p$) away'' from the classical (uniform strict) ellipticity, 
  $\Delta_{2}
  (A)>0$, 
hence its name.

The matrix $A$ is {\sl real} elliptic if and only if it is $p$-elliptic for every $p>1$. Moreover, for any bounded complex $A$ we have $\Delta_1(A)\leq0$ with the equality precisely for real positive semidefinite $A$.

For $A\in\cA(\Omega)$ we also set
\begin{equation}
\label{eq: potrosijo}
\mu(A):={\rm ess}\inf\Re\frac{\sk{A(x)\xi}{\xi}}{\left|\sk{A(x)\xi}{\bar\xi}\right|}\,.
\end{equation}
By this we mean that $\mu(A)={\rm ess}\inf\f$, where $\f:\Omega\rightarrow\R\cup\{+\infty\}$ is defined by
$$
\f(x)=\inf\Re\frac{\sk{A(x)\xi}{\xi}}{\left|\sk{A(x)\xi}{\bar\xi}\right|}\,
$$
and the above infimum runs over all $\xi\in\C^n$ for which $\sk{A(x)\xi}{\bar\xi}\ne0$.

The importance of $\mu(A)$ lies in the bilateral estimate  
$\Delta_p(A)\sim\mu(A)-|1-2/p|$ (Proposition~\ref{p: reci dje si}), therefore
the 
key condition \eqref{eq: aferim jano} is equivalent to
\begin{equation}
\label{eq: nemam ti kad}
|1-2/p|<\mu(A).
\end{equation}
The advantage of the inequality \eqref{eq: nemam ti kad} over \eqref{eq: aferim jano} is that it ``separates'' $A$ from $p$.

In view of the basic assumptions \eqref{eq: elli} and \eqref{eq: ne mogu se kontrolirati}, the
quantity $\mu(A)$ is trivially bounded from below: 
for $A\in\cA_{\lambda,\Lambda}(\Omega)$ we have
$
\mu(A)\geq\lambda/\Lambda\,.
$
Therefore \eqref{eq: aferim jano} and \eqref{eq: nemam ti kad} are already satisfied if $|1-2/p|<\lambda/\Lambda$. 
We also have the sharp universal upper bound $\mu(A)\leq1$, which follows from considering $\xi\in\R^n$ in  
\eqref{eq: potrosijo}. To summarize,
\begin{equation}
\label{eq: vrse su u moru}
\frac{\lambda_A}{\Lambda_A}\leq\mu(A)\leq1.
\end{equation}

We remark that a similar yet weaker condition than \eqref{eq: aferim jano}, namely $\Delta_p(A)\geq0$, was formulated in a different form by Cialdea and Maz'ya in \cite[(2.25)]{CM}, see Remark~\ref{r: pletnev schumann op. 17}. It was a result of their study of a condition on sesquilinear forms known as $L^{p}$-dissipativity. We arrived at \eqref{eq: aferim jano}, and thus at $\Delta_p(A)\geq0$, from another direction (bilinear embeddings and generalized convexity of power functions); see Remark \ref{r: tovar grlo stima} for a summary.

Finally, when $A,B$ are two accretive matrices, we denote
$$
\mu(A,B):=\min\left\{\mu(A),\mu(B)\right\}
\quad {\rm and}\quad
\Delta_p(A,B):=\min\left\{\Delta_p(A),\Delta_p(B)\right\}\,.
$$

\subsection{Bilinear embedding for pairs of complex accretive matrices}
\label{s: bilincomplex}

\begin{thm}
\label{t: bilincomplex} 
Let $p>1$. Suppose that $A,B\in\cA_{\lambda,\Lambda}(\R^n)$ satisfy $\Delta_{p}(A,B)>0$.
Then for all $f,g\in C_c^\infty(\R^n)$ we have
\begin{equation}
\label{eq: bilincomplex}
\int_0^\infty\int_{\R^n} |\nabla  P^{A}_{t}f(x)|\,|\nabla  P^{B}_{t}g(x)|\,\,dx\,dt\leq 
\frac{20}{\Delta_p(A,B)
}\cdot\frac\Lambda\lambda
\nor{f}_p\nor{g}_q\,.
\end{equation}
\end{thm}

\noindent
The proof will be given in Section~\ref{s: sonno}.

As noted earlier, when we restrict ourselves to real accretive matrices, the condition $\Delta_p(A,B)>0$ is automatically fulfilled, so \eqref{eq: bilincomplex} holds for the full range of exponents $p\in(1,\infty)$. Hence Theorem~\ref{t: oktyabr'} is a special case ($A,B$ equal and real) of Theorem~\ref{t: bilincomplex}. One vital difference between the real and the complex case is that in the latter the condition $\Delta_p(A,B)>0$, and hence the desired type of convexity of the {\it Bellman function} that we use, does not hold for all $p\in(1,\infty)$.
The other difference is that in the former case the semigroup $(P^{A}_{t})_{t>0}$ is bounded for all $p\in[1,\infty]$ (we used this when proving Theorem~\ref{t: oktyabr'}, see \cite[p. 2826]{DV-Kato}), while in the complex case this is false (see Section~\ref{s: saraj}). 

\medskip
We call Theorem~\ref{t: bilincomplex} the (dimension-free) {\it bilinear embedding theorem} for complex accretive matrices, because it implies that the map $\text{\got I}$, defined for $f,g,\in C_c^\infty(\R^n)$ by 
$$
\left[\text{\got I}(f,g)\right](x,t):=\sk{\nabla  P^{A}_{t}f(x)}{\nabla  P^{B}_{t}\bar{g}(x)}_{\C^n},
$$
extends to a bounded bilinear map $\text{\got I}: L^{p}(\R^n)\times L^{q}(\R^n)  \rightarrow  L^1(\R_+^{n+1})$ with explicit norm estimates that only depend on $\Delta_p(A,B)$ and the ellipticity constants, and do not depend on the dimension $n$. This feature of Theorem \ref{t: bilincomplex} may be of independent interest, since it is in keeping with many results in harmonic analysis where the emphasis lies on the independence on the dimension (e.g., applications to infinite-dimensional analysis). A notable example is a theorem of Stein on the Riesz transforms \cite{SteinRiesz}, after which followed an array of generalizations. One of them is a theorem by A. Volberg and the second-named author of the present paper \cite{DV} which was proven by a reduction to a bilinear estimate akin to \eqref{eq: bilincomplex}, but with the Poisson semigroup instead of the heat one. The current authors later obtained a considerably more general result \cite{CD} by further developing this method.

Thus bilinear embedding is a type of estimate that has been instrumental in proving a variety of sharp results, e.g. Riesz transform estimates \cite{PV,NV,DV-AB,DV,PSW,CD,DP} and recently also general spectral multiplier results \cite{CD-mult, CD-OU}. It appears to us that in the absence of regularity of the coefficients our method cannot be directly adapted for proving boundedness of Riesz transforms; this would require a new idea.

\subsection{Sharpness} 
\label{s: vedno znova lotu}

In general, when $\Delta_p(A,B)<0$ the dimension-free bilinear embedding \eqref{eq: bilincomplex} fails.
The precise formulation of this statement is Proposition~\ref{p: reki begalci}.

Throughout the paper we will, for $p\in(1,\infty)$, often use the notation
\begin{equation*}
\label{eq: tula}
\phi_p:=\arccos|1-2/p|\,.
\end{equation*}
This is known to be the optimal angle of holomorphy of symmetric contraction semigroups on $L^{p}$, see \cite{LP, Kr, CD-mult,Haase2016}. Its complementary angle, $\pi/2-\phi_p$, was recently proven by the present authors to be the optimal angle in the holomorphic functional calculus for generators of symmetric contraction semigroups \cite{CD-mult}.
Proposition~\ref{p: reki begalci} is another 
sharp result which features $\phi_p$.

For each $A,B\in\cA(\R^n)$ let
\begin{equation}
\label{eq: rocca di calascio}
N_p(A,B):=
\sup_{f,g\in C_c^\infty\setminus\{0\} 
}
\frac1{\nor{f}_p\nor{g}_q}
\int_0^\infty\int_{\R^n} |\nabla  P^{A}_{t}f(x)|\,|\nabla  P^{B}_{t}g(x)|\,\,dx\,dt 
\,.
\end{equation}
With this notation, the conclusion \eqref{eq: bilincomplex} of Theorem~\ref{t: bilincomplex} can be restated as:

\medskip
{\it Fix $p,\lambda,\Lambda$. For any $\Delta>0$ there exists an explicit $
C(\Delta,\lambda,\Lambda)>0$ such that}
$$
\sup\Mn{ N_p(A,B)}{A,B\in\cA_{\lambda,\Lambda}(\R^n), \Delta_p(A,B)\geq\Delta, n\in\N}\leq C(\Delta,\lambda,\Lambda)<\infty\,.
$$
 
\medskip
\noindent
We show in Proposition~\ref{p: reki begalci} below that for $\Delta<0$ this conclusion is false, even if the supremum is taken over a smaller subfamily of complex rotations of identity matrices. 

Notice that if $A=e^{i\phi}I_n$ then $\lambda_A=\cos\phi$, $\Lambda_A=1$, $\mu(A)=\cos\phi$ and moreover $\Delta_p(A)=\cos\phi-\cos\phi_p=\Delta_p(A^*)<0$ if $\phi_p<|\phi|<\pi/2$. See also \eqref{eq: perasovic}.

\begin{prop}
\label{p: reki begalci}
Fix an arbitrary $p\in(1,\infty)\backslash\{2\}$ and 
$\phi\in(\phi_p,\pi/2)$. 
For any $n\in\N$ write $A_n=e^{i\phi}I_n$. 
Then
$$
\sup_{n\in\N}N_p(A_n,A_n^*)
=\infty\,.
$$
\end{prop}
This result follows from determining the $L^{p}$-contractivity domain of the classical heat semigroup with complex time. See Section \ref{s: zavratnica} and page \pageref{begalci 2015}.

\subsection{Square function estimates dominate the bilinear embedding}
\label{s: mota}

We have just formulated the failure of the dimension-free estimate \eqref{eq: bilincomplex} for pairs $(A,B)$ of matrices for which $\Delta_p(A,B)<0$, and saw that this was obtained by taking the subfamily of pairs $(A,A^*)$. The initial interest in Theorem~\ref{t: bilincomplex} arose from the quest to extend Theorem~\ref{t: oktyabr'} to nonreal $A$, as well as to study the case $B=A^*$ for real or complex $A$. The latter question was posed to the second-named author of the present paper by P. Auscher in July of 2011; the rationale behind this question is discussed here.

\medskip
Results announced in Section~\ref{s: saraj} tell us that the {\sl bilinear embedding is a sufficient condition for semigroup estimates}. Now we explain that it is also a {\sl necessary condition for square function estimates}.

Obviously, bilinear integrals are dominated by the {\it vertical square function}:
\begin{equation}
\label{eq: hor MVD}
\int_0^\infty\int_{\R^n} |\nabla  e^{-tL_{A}}f(x)|\,|\nabla  e^{-tL_{B}}g(x)|\,dx\,dt\leq\Nor{G_1^{L_{A}}f}_p\Nor{G_1^{L_{B}}g}_q\,,
\end{equation}
where
$$
G_1^{L}u(x) := \left(\int_{0} ^{\infty}\big|\nabla e^{-tL}u(x)\big|^2\,dt\right)^{1/2}\,.
$$
It was proven by Auscher \cite{A} that, even if $A\in\cA_{\lambda,\Lambda}(\R^n)$ is real, $G_1^{L_{A}}$ is $L^{p}$ bounded only for a certain range of $p$'s which depends on $p,n,\lambda,\Lambda$. See \cite[Corollaries 6.3-6.7]{A} or \cite[Proposition~1.2]{AHM}. 
On the other hand, Theorem~\ref{t: oktyabr'} holds for all $p\in(1,\infty)$. This is somewhat surprising, since the bilinear integral and the vertical square function are only ``two H\"older inequalities apart'', cf. \eqref{eq: hor MVD}. On top of that, Theorem~\ref{t: oktyabr'} features dimension-free constants. To balance this, in \cite[Theorem~3.1.(2)]{AHM} the $L^{p}$ estimates were obtained for {\it conical} square functions \eqref{eq: conical}. The range of admissible $p$'s there is related to the uniform boundedness of $\left(P_{t}^{A}\right)_{t>0}$ on $L^{p}$, which in case of real $A$ equals $(1,\infty)$. The estimates in \cite[Theorem~3.1]{AHM} depend on $n$. It would be interesting to investigate the relation between dimension-free estimates of conical square functions associated with complex matrices on one hand and the semigroup contractivity or $p$-ellipticity on the other.

Bilinear integrals that we consider in Theorem~\ref{t: bilincomplex} naturally correspond to 
the following two {\it nontangential} or {\it conical square functions} associated with $L=L_{A}$: 
\begin{eqnarray}
&&\label{eq: conical}
\g_1^L(u)(x) =
{\displaystyle
\left(
\iint_{V_x}\left|\nabla(e^{-tL}u)(y)\right|^2\,\frac{dy\,dt}{t^{n/2}}
\right)^{1/2}},\\ 
&&\g_2^L(u)(x)=
{\displaystyle
\left(
\iint_{V_x}\left|L^{1/2}(e^{-tL}u)(y)\right|^2\,\frac{dy\,dt}{t^{n/2}}
\right)^{1/2}}
,\nonumber
\end{eqnarray} 
where $V_x$ is the cone $\Mn{(y,t)\in\R^{n}\times(0,\infty)}{|x-y|<\sqrt t}$. Considerable attention has been devoted to studying $L^{p}$ properties of $\sqrt{\g_1^L(u)^2+\g_2^L(u)^2}$, see \cite[Section~6.2]{A}. The boundedness of $\g_1$ alone was treated in \cite[Proposition~1.3]{AHM} for the case where $A$ is real.

The functionals $\g_1^L,\g_2^L$ are related to the bilinear estimates \eqref{eq: bilincomplex} with $B=A^*$. 
Indeed, owing to a known averaging trick, see \cite{AHM} and the references therein, we have estimates
\begin{equation}
\label{eq: pecenice}
\aligned
\left|\int_0^\infty\int_{\R^{n}} \sk{A\nabla  P_t^Af(x)}{ \nabla  P_t^{A^*}g(x)}_{\C^{n}}\,dx\,dt\right|& \, \leqsim_n \Nor{\g_j^{L_A}(f)}_p\Nor{\g_j^{L_A^*}(g)}_q
\endaligned
\end{equation}
for $j=1,2$. The inequality with $j=1$ still holds true if in the left-hand side we put the modulus inside the integral. We also have \eqref{eq: pecenice} if in the right-hand side we  replace $\g_2^L$ by its {\it vertical} counterpart $G_2^L$, defined by 
$$
G_2^{L}u(x) := \left(\int_{0} ^{\infty}\big|L^{1/2} e^{-tL}u(x)\big|^2\,dt\right)^{1/2}\,.
$$

\subsection{Semigroup estimates} 
\label{s: saraj}
The $L^{p}$ estimates of semigroups generated by elliptic operators in divergence form have long known to be of 
major importance \cite{Ba, A, O, P, CM, CM2014}. A result of Auscher \cite[Corollary~3.6]{A} asserts that if $|1/2-1/p|\leq 1/n$ then $(e^{-tL_{A}})_{t>0}$ is bounded on $L^{p}(\R^n)$. Hofmann, Mayboroda and McIntosh proved in \cite{HMMcI} that this condition is sharp in terms of $n$, in the sense that if $|1/2-1/p|> 1/n$ then $(e^{-tL_{A}})_{t>0}$ is not bounded on $L^{p}(\R^n)$ for some $A\in\cA(\R^n)$.

The bilinear embedding associated with $(A,A^*)$ implies uniform $L^p$ boundedness of $\left(P_t^A\right)_{t>0}$. Indeed, for $A,B\in\cA_{\lambda,\Lambda}(\R^n)$ and $f,g\in C_c^\infty(\R^n)$ define $\f(s):=\sk{P_{s}^{A}f}{P_{s}^{B}g}_{L^2}$ for $s>0$. Then $\f'(s)=-\sk{(A+B^*)\nabla P_{s}^{A}f}{\nabla P_{s}^{B}g}_{L^{2}}$.
By the injectivity of $L_{A}$ and the analyticity of the associated semigroup, $\nor{P_s^Af}_2\rightarrow0$ as $s\rightarrow\infty$. Therefore $\f(s)\rightarrow0$ as $s\rightarrow\infty$. Hence
$$
\f(t)=-\int_t^\infty\f'(s)\,ds.
$$
Consequently,
$$
\aligned
\left|\sk{P_{t}^{A}f}{P_{t}^{B}g}_{L^{2}}\right|&
\leq 2\Lambda\int_0^\infty\sk{|\nabla P_s^Af|}{|\nabla P_s^Bg|}_{L^{2}}\,ds
.
\endaligned
$$
Keeping in mind the notation \eqref{eq: rocca di calascio}, this implies that 
\begin{equation}
\label{eq: besko}
\sup_{t>0}
\nor{\left(P_{t}^{B}\right)^*P_{t}^{A}}_{\cB(L^{p}(\R^n))}
\leqslant 2\Lambda N_{p}(A,B). 
\end{equation}
When $B=A^*$ one has $\left(P^{B}_{t}\right)^{*}=P^{A}_{t}$. Therefore Theorem~\ref{t: bilincomplex} immediately gives that the semigroup $(P^{A}_{t})_{t>0}$ is uniformly bounded on $L^{p}$ when $\Delta_{p}(A,A^{*})>0$, which is in turn, by Corollary~\ref{c: garrison} {\it \ref{eq: haeri}.)}, equivalent to $\Delta_{p}(A)>0$. Actually, by Theorem \ref{t: Aufnahmebetriebsartenwaehler} below, when $\Delta_{p}(A)>0$ the semigroup $\left(P^{A}_{t}\right)_{t>0}$ turns out to be {\sl contractive} on $L^p$.

The problem of characterizing $L^{p}$ contractivity of semigroups generated by elliptic divergence-form operators has a long history. Early results include the 1959 paper by Agmon, Douglis and Nirenberg \cite{ADN}. For a rather recent summary of references on this topic the reader is advised to consider Cialdea and Maz'ya \cite{CM} and \cite[pp. 71--72]{CM2014}. Their 2005 paper \cite{CM} was a major step towards understanding the problem for general $(P^{A}_{t})_{t>0}$; see also their monograph \cite[Chapter 2]{CM2014}. Let us take a closer look at the contractivity results there. 

Assuming the notation \eqref{eq: verve}, under additional assumptions that either: 
\begin{itemize}
\item
$\Omega\subset\R^n$ is a bounded domain with sufficiently regular boundary \cite[p. 1087]{CM}, the entries of $A$ belong to the class $C^1(\overline\Omega)$, and $\Im A$ is symmetric, that is, $(\Im A)_{\sf a}=0$; or else
\item
$A$ is constant and $\Omega$ contains balls of arbitrarily large radius, 
\end{itemize}
Cialdea and Maz'ya proved that the contractivity of $(P^{A}_{t})_{t>0}$ on $L^{p}(\Omega)$ is equivalent to $\Delta_p(A_{\mathbf s})\geq0$. See \cite[Theorems 5,2,3]{CM}, \cite[Theorem~2.23]{CM2014} and
Proposition~\ref{p: trazi se morricone} below. 
Cialdea \cite[p. 74]{C} asked about generalizing the results from \cite{CM} beyond the restrictions posed by the cited smoothness and symmetry conditions. In response to these questions, we extend in Theorem~\ref{t: Aufnahmebetriebsartenwaehler} the characterization by Cialdea and Maz'ya to the case when
\begin{itemize}
\item
$\Omega\subset\R^n$ is an arbitrary open set, 
\item
$A$ is not necessarily smooth, and 
\item
$\div(\Im A)_{\sf a}^{(k)}=0$ for all $k\in\{1,\hdots,n\}$, 
but not necessarily $(\Im A)_{\sf a}=0$. 
\end{itemize}

\noindent
Here $\div\,W^{(k)}$ is the distributional divergence of the $k$-th column of a matrix $W$ with entries in $L_{loc}^1(\Omega)$. 
\label{s: svoju rabotu ceram}
We show that the last remaining case, that is, when for some $k$ we have 
$\div(\Im A)_{\sf a}^{(k)}\ne0$, 
is fundamentally different, because then the condition $\Delta_p(A_{\mathbf s})\geq0$ is in general {\it not} equivalent to the contractivity of $(P^{A}_{t})_{t>0}$ on $L^{p}(\Omega)$, not even for $A\in C^\infty(\R^n)$.

We also prove that $\Delta_p(A)\geq0$ is a sufficient condition for $L^p$ contractivity which is devoid of any smoothness or symmetry assumptions on $A$ or geometric conditions on $\Omega$.

\medskip
Now we state our result. 

\begin{thm}
\label{t: Aufnahmebetriebsartenwaehler}
Suppose that $n\in\N$, $\Omega\subset\R^n$ is open, $A\in\cA(\Omega)$ and $p>1$. 
Consider the following statements:
\begin{enumerate}[(a)]
\item
\label{Imshi} 
$\Delta_p(A)\geq0$;
\item
\label{Wara} 
$(P^{A}_{t})_{t>0}$ extends to a contractive semigroup on $L^{p}(\Omega)$;
\item
\label{eq: Kidbuhom}
$\Delta_p(A_{\sf s})\geq0$.
\end{enumerate}
Then: 
\begin{itemize}
\item
{\it(\ref{Imshi})} $\Rightarrow$ {\it(\ref{Wara})} 
$\Rightarrow$ {\it(\ref{eq: Kidbuhom})};
\vskip 2pt
\item
if\, $\div(\Im A)_{\sf a}^{(k)}=0$ for all $k\in\{1,\hdots,n\}$  
then {\it(\ref{Wara})} $\Leftrightarrow$ {\it(\ref{eq: Kidbuhom})};
\vskip 2pt

\item
if \,$\div(\Im A)_{\sf a}^{(k)}\ne0$ for some $k\in\{1,\hdots,n\}$ 
then, in general, {\it(\ref{eq: Kidbuhom})} $\not\Rightarrow$ {\it(\ref{Wara})}.
\end{itemize}
\end{thm}

\noindent
We prove Theorem \ref{t: Aufnahmebetriebsartenwaehler} in Section \ref{masamune}.
Alternative descriptions of conditions $\Delta_p(A)\geq0$ and $\Delta_p(A_{\sf s})\geq0$ are contained in Propositions~\ref{p: sahbaz} and \ref{p: trazi se morricone}.

\medskip
\noindent
{\bf Comments.} 
The sufficiency part of Theorem~\ref{t: Aufnahmebetriebsartenwaehler} thus complements the sharp results by Auscher \cite{A} as follows: for any $A\in\cA(\R^n)$,
 \begin{itemize}
 \item
 (Auscher \cite{A}) 
 \hskip 4.3pt 
 if $|1-2/p|\leq\, 2/n$\,\hskip 5.3pt then $(P^{A}_{t})_{t>0}$ is bounded on $L^{p}(\R^n)$;
 \vskip 2pt
 \item
 (Theorem~\ref{t: Aufnahmebetriebsartenwaehler}) 
 if $|1-2/p|\leq 
 \mu(A)$ then $(P^{A}_{t})_{t>0}$ is contractive on $L^{p}(\R^n)$.
 \end{itemize}
Compare also with Bakry \cite[Th\'eor\`eme 7]{Ba}.

The admissible (and optimal) range of $p$'s in Auscher's \cite{A} above-cited result shrinks to $\{2\}$ as $n\rightarrow\infty$.
Regarding the {\sl contractivity}, given $n\in\N$, the largest set $J_n\subset(1,\infty)$ so that 
$P^{A}_{t}$ is contractive on $L^{p}(\R^n)$ for any $p\in J_n$ and any 
$A\in\cA(\R^n)$, is just $J_n=\{2\}$.
Counterexamples are again provided by Theorem~\ref{t: tgv bordeaux-cdg}: 
given $p\ne2$, it suffices to take $A=\exp(i\phi)I_n$ with $\phi\in(\phi_p,\pi/2)$.
On the other hand, the condition from Theorem~\ref{t: Aufnahmebetriebsartenwaehler} is dimension-free. 

Assume, as in \cite[p. 1087]{CM}, that $\Omega$ is a bounded domain with $\partial\Omega\in C^2$ and that the entries of $A$ belong to $C^1(\overline\Omega)$. By elliptic regularity \cite{ADN} we have that for any $p\in(1,\infty)$ the semigroup $\left(P_t^A\right)_{t>0}$ extends to an analytic semigroup on $L^p(\Omega)$ and the domain of its generator is $W_0^{1,p}(\Omega)\cap W^{2,p}(\Omega)$. 
For more details see \cite[Theorem~6.3.4]{E} for the case $p=2$ and \cite[Section 3]{Lun} 
for arbitrary $p\in(1,\infty)$. Thus Theorem~\ref{t: Aufnahmebetriebsartenwaehler} 
generalizes the $L^{p}$-contractivity result of Cialdea and Maz'ya \cite[Theorem~5]{CM}.

A minor modification of Example~1 from \cite{CM}, so as to include elliptic matrices, shows that {\it(\ref{Wara})} of Theorem~\ref{t: Aufnahmebetriebsartenwaehler} does not imply {\it(\ref{Imshi})}, not even for constant $A$.

Section~\ref{s: pilot} and Theorem~\ref{t: Aufnahmebetriebsartenwaehler} also imply a known result, formulated in Ouhabaz \cite[Theorem~4.28]{O}, that when $A$ is real, the semigroup $(P^{A}_{t})_{t>0}$ is contractive on $L^{p}(\Omega)$ for all $1<p<\infty$. More generally, \cite[Theorem~4.29]{O} implies that if 
\begin{equation}
\label{eq: spat' hochetsya}
\Im A \text{ is purely antisymmetric and }\div(\Im A)_{\mathsf a}^{(k)}=0
\end{equation}
then $(P^{A}_{t})_{t>0}$ extends to a contraction semigroup on $L^{p}(\Omega)$. This is again a special case of our Theorem~\ref{t: Aufnahmebetriebsartenwaehler}, because $\Im A_{\mathsf s}=0$ is for $A\in\cA(\Omega)$ equivalent to 
$\Delta_\infty(A_{\mathsf s})\geq0$, which implies (see Section \ref{s: pilot}) that 
$\Delta_\infty(A_{\mathsf s})=0$ and hence $\Delta_p(A_{\mathsf s})>0$ for all $p\in(1,\infty)$ (Corollary \ref{c: niagara} and Proposition \ref{p: reci dje si}).
The appearance of the ``limit case $p\rightarrow\infty$'' of {\it(\ref{eq: Kidbuhom})}, namely,  $\Delta_\infty(A_{\mathsf s})\geq0$, should in view of Theorem \ref{t: Aufnahmebetriebsartenwaehler}  not come as a surprise. A complementary reason to expect it is that the above-cited contractivity result contained in \cite[Theorem~4.29]{O} follows, by complex interpolation, from a stronger result, see \cite[Corollaire~2.2]{ABBO} or \cite[Corollary~4.12]{O}, which asserts that \eqref{eq: spat' hochetsya} in fact {\it characterizes} the $L^{\infty}$-contractivity of $\left(P_t^A\right)_{t>0}$.

In the special case of $A$ being smooth on a bounded domain, the authors of \cite{CM} indirectly 
prove the implication {\it(\ref{Imshi})} $\Rightarrow$ {\it(\ref{Wara})} of Theorem~\ref{t: Aufnahmebetriebsartenwaehler}; see Proposition~\ref{p: sahbaz} as well as their Corollary~4, Theorem~3 and the proof of Theorem~5. 
In our proof of {\it(\ref{Imshi})} $\Rightarrow$ {\it(\ref{Wara})} however, as said before, no smoothness or symmetry of $A$ is assumed, and $\Omega$ is allowed to be an arbitrary open set.

Finally, we saw that $p$-ellipticity implies both the dimension-free bilinear embedding (Theorem \ref{t: bilincomplex}) and the semigroup contractivity (Theorem~\ref{t: Aufnahmebetriebsartenwaehler}). We also know that bilinear embedding implies boundedness of the semigroup \eqref{eq: besko}. Thus it would be natural to inquire about a direct connection between {\it dimension-free} bilinear embeddings on $L^p\times L^q$ and semigroup {\it contractivity} on $L^p$.

\subsection{Organization of the paper}
Section~\ref{s: Prokofjev 5. simfonija 3. stavak Karajan 1969} serves the purpose of collecting in one spot most of the definitions and facts indispensable for this paper. In Section~\ref{s: sokolov hortus musicus} we sketch the main ideas behind our proofs, devoting particular attention to explaining the heat-flow-Bellman-function method. In Section~\ref{s: chakrulo} we show how integration by parts of the flow associated with the function $\Phi$ helps identify the fundamental convexity requirement on $\Phi$. 
In Section~\ref{s: garanča} we define the Bellman function and show that it possesses the desired convexity. By considering the Hessians of power functions in one complex variable we explain how the condition \eqref{eq: aferim jano} was born.
In Section~\ref{s: sonno} we complete the proof of the bilinear embedding (Theorem~\ref{t: bilincomplex}).
In Section~\ref{masamune} we prove our result on the contractivity of semigroups (Theorem~\ref{t: Aufnahmebetriebsartenwaehler}).
Finally, Appendix is a technical part that provides a regularization argument used in the proof of Theorem~\ref{t: bilincomplex}.

\section{More notation and preliminaries}
\label{s: Prokofjev 5. simfonija 3. stavak Karajan 1969}

For $a_1,a_2>0$ we write $a_1\geqsim a_2$ if there is a constant $C>0$ such that $a_1\geqslant C a_2$. Similarly we define $a_1\leqsim a_2$. If both $a_1\geqsim a_2$ and $a_1\leqsim a_2$ then we write $a_1\sim a_2$.

We will denote $\C_+=\mn{\zeta\in\C}{\Re\zeta>0}$. 
Let $n\in\N$. If $z=(z_1,\hdots,z_n)\in\C^n$, we denote
$\Re z=(\Re z_1,\hdots,\Re z_n)$, $\Im z=(\Im z_1,\hdots,\Im z_n)$ and
$\bar z=(\bar z_1,\hdots,\bar z_n)$.
If also $w\in\C^n$, we write
$$
\sk{z}{w}_{\C^n}=\sum_{j=1}^nz_j\overline w_j\,
$$
and $|\xi|^2=\sk{\xi}{\xi}_{\C^n}$.
When the dimension is obvious, we sometimes omit the index $\C^n$ and only write $\sk{z}{w}$.
When both $z$ and $w$ belong to $\R^n$, we sometimes emphasize this by writing $\sk{z}{w}_{\R^n}$.
This should not be confused with the standard pairing
\begin{equation}
\label{eq: si mislyat}
\sk\f\psi=\int_{\R^n}\f\bar\psi,
\end{equation}
where $\f,\psi$ are complex functions on $\R^n$ such that the above integral makes sense.
All the integrals in this paper are taken over the Lebesgue measure $m$, therefore we will mostly write 
them without $dm$ at the end.

If $x_1,\hdots,x_n$ are the coordinates on $\R^n$, we define, initially on $C_c^\infty(\R^n)$ or $\cS(\R^n)$, by 
$$
\Delta_n=\sum_{j=1}^n\frac{\pd^2 }{\pd x_j^2}
$$
the {\it Laplace operator} on $\R^n$. When the underlying dimension  
is clear, we simply write $\Delta$. The same symbol will also denote the 
negative of the generator of the classical heat semigroup on $L^{p}(\R^n)$.

When $F=F(x_1,\hdots,x_m)\in C^2(U)$ for some open $U\subset \R^m$, we 
introduce the {\it Hessian matrix} of $F$, calculated at $x\in U$:
$$
{\rm Hess}(F;x)
=\left[\pd^2_{x_ix_j}F(x)\right]_{i,j=1}^m\,.
$$

Let $\C^{n,n}$ be the space of all complex $n\times n$ matrices. For $M\in\C^{n,n}$ denote its conjugate transpose by $M^{*}$ and define its {\it symmetric part} $M_{\sf s}$ and {\it antisymmetric part} $M_{\sf a}$ by
\begin{equation}
\label{eq: verve}
M_{\sf s}:=\frac{M+M^T}2,
\quad\quad
M_{\sf a}:=\frac{M-M^T}2.
\end{equation}

\label{ah moj aljo}
Write $I_{\R^n}$ for the identity matrix on $\R^n$.
If $M_1,M_2\in\R^{m,m}$ then let $M_1\oplus M_2$
denote the $2m\times 2m$ block-diagonal matrix having 
$M_1,M_2$ (in this order) on the diagonal. 
If $f,g$ are complex functions on some sets $X,Y$ respectively, then $f\otimes g$ is the abbreviation for the function on $X\times Y$ mapping $(x,y)\mapsto f(x)g(y)$.

Let $L^\infty(\Omega\rightarrow\C^{n,n})$ denote the space of complex matrix functions on $\Omega$ with entries in $L^\infty(\Omega)$ and by $\cB(X)$ the space of bounded linear operators on a Banach space $X$.

\subsection{Identification operators}
\label{s: grand sonata luganskij}
We will explicitly identify $\C^n$ with $\R^{2n}$. 
For each $n\in\N$ consider the operator
$\cV_{n}:\C^{n}\rightarrow\R^{n}\times\R^{n}$, defined by 
$$
\cV_{n}(\alpha+i\beta)=
(\alpha,\beta).
$$
 One has, for all  $z,w\in\C^n$,
 \begin{equation}
\label{eq: angelis - romance - hatmulin}
\Re\!\sk{z}{w}_{\C^n}=\sk{\cV_n(z)}{\cV_n(w)}_{\R^{2n}}
\,.
\end{equation}
If $(\omega_{1},\omega_{2})\in\C^{n}\times\C^{n}$ then
$
\cV_{2n}(\omega_{1},\omega_{2})
=(\Re\omega_1, \Re\omega_2, \Im\omega_1, \Im\omega_2)
\in(\R^n)^4.
$
On $\C^n\times\C^n$ define another identification operator $\cW_{2n}:\C^{n}\times\C^{n}\rightarrow (\R^{n})^4$, 
$$
\cW_{2n}(\omega_{1},\omega_{2})
=(\cV_n(\omega_1),\cV_n(\omega_2))
=(\Re\omega
_1, \Im\omega_1, \Re\omega_2, \Im\omega_2). 
$$
When the dimensions of the spaces on which the identification operators act is clear, we will sometimes omit the indices and instead of $\cV_n,\cW_m$ only write $\cV,\cW$.

Given a matrix $D=[d_{ij}]_{i,j}\in\R^{M,N}$ and $n\in\N$, we let $D\otimes I_{\R^n}\in\R^{Mn,Nn}$ be the 
Kronecker product of $D$ with the identity on $\R^n$, that is,
$$
D\otimes I_{\R^n}:=[d_{ij}\cdot I_n]_{i,j}\,.
$$
For example, 
$$
[a\ b]\otimes I_{\R^3}=
\left[
\begin{array}{cccccc}
 a & & & b & &   \\
&  a & & & b &    \\
& &  a & & & b   
\end{array}
\right]
\hskip 20pt
\text{and }
\hskip 20pt
\left[
\begin{array}{cc}
 a & b   \\
c & d
\end{array}
\right]
\otimes I_{\R^2}=
\left[
\begin{array}{cccc}
 a & & b  &   \\
  & a &  & b  \\
 c & & d  &   \\
  & c &  & d 
\end{array}
\right]\,.
$$

If $A=[a_{ij}]_{i,j=1}^n\in\C^{n,n}$
then $\Re A:=[\Re a_{ij}]_{i,j=1}^n$ and $\Im A:=[\Im a_{ij}]_{i,j=1}^n$. 
We shall frequently need the following derived real $(2n)\times(2n)$ matrix:
$$
\cM(A)=\left[
\begin{array}{rr}
\Re A  & -\Im A\\
\Im A  & \Re A
\end{array}
\right]\,.
$$
Its significance stems from the formula 
\begin{equation}
\label{eq: zubicki - perpetuum mobile - hatmulin}
\cV(A\xi)=\cM(A)\cV(\xi)  \hskip 50pt \text{ for all }\xi\in\C^n.
\end{equation}
We can view $\cM$ as a mapping $\C^{n,n}\longrightarrow\R^{2n,2n}$. 
Observe that $\cM(A^*)=\cM(A)^T$ and $\cM(AB)=\cM(A)\cM(B)$. We derive from \eqref{eq: angelis - romance - hatmulin} and \eqref{eq: zubicki - perpetuum mobile - hatmulin} the useful identities
\begin{equation}
\label{eq: profana}
\aligned
\Re\sk{A\xi}{\eta}_{\C^n}&=
\sk{\cM(A)\cV(\xi)}{\cV(\eta)}_{\R^{2n}}\\
\Im\sk{A\xi}{\eta}_{\C^n}
&=\sk{\cM(A)\cV(\xi)}{\cV(i\eta)}_{\R^{2n}}.
\endaligned
\end{equation}

\subsection{Generalized Hessians and generalized convexity}
\label{s: cajkovskij grand sonata}
The objects and notions defined in this section will appear throughout the paper. 
While in principle one definition would suffice, for the sake of clarity we treat the case of functions defined on $\C$ and associated with a single matrix, and the case of functions defined on $\C^2$ and associated with a pair of matrices, separately.

\subsubsection*{One-dimensional case}
Take $A,B\in \C^{n, n}$. 
Suppose that $F:\C\rightarrow \R$ is smooth, $s\in\C$ and  $\xi\in\C^n$. We set
\begin{equation}
\label{eq: ulica marata 5}
H^{A}_{F}[s;\xi]= \sk{\left[{\rm Hess}_{\cV_1}(F;s)\otimes I_{\R^{n}}\right]
\cV_{n}(\xi)}{\cM(A)\cV_{n}(\xi)}_{\R^{2n}},
\end{equation}
where ${\rm Hess}_{\cV_1}(F;s)={\rm Hess}(F\circ\cV_1^{-1};\cV_1(s))$. 
In block notation,  with $\cH=\R^n$, 
$$
H^{A}_{F}[s;\xi]=
\sk{ {\rm Hess}_{\cV_1}(F;s)
\left[\begin{array}{c} \Re\xi \\\Im\xi\end{array}\right]}
{\left[\begin{array}{cr}\Re A & -\Im A \\ \Im A & \Re A\end{array}\right]
\left[\begin{array}{c}\Re\xi \\ \Im\xi\end{array}\right]
}_{\cH^2}.
$$
We say that $F$ is {\it convex with respect to $A$} if $H_F^A[s;\xi]\geq 0$ for all $s,\xi$. 

\subsubsection*{Two-dimensional case}
Similarly, if $\Phi:\C^{2}\rightarrow \R$ is smooth, $v\in\C^{2}$ and $\omega=(\omega_{1},\omega_{2})\in \C^{n}\times\C^{n}$, define
\begin{equation}
\label{eq: tetastevka}
 H^{(A,B)}_{\Phi}[v;\omega]=
 \sk{
 \left[
    {\rm Hess}_{\cW_2}(\Phi;v)\otimes I_{\R^n}
 \right]
 \cW_{2n}(\omega)}
 {\left[\cM(A)\oplus \cM(B)\right]\cW_{2n}(\omega)}_{\R^{4n}},
\end{equation}
where ${\rm Hess}_{\cW_2}(\Phi;v)={\rm Hess}(\Phi\circ\cW_2^{-1};\cW_2(v))$. 
In block notation,  with $\cH=\R^n$, 
$$
H^{(A,B)}_{\Phi}[v;\omega]=
\sk{ {\rm Hess}_{\cW_2}(\Phi;v) 
\left[\begin{array}{c}\Re\omega_1 \\ \Im\omega_1 \\ \Re\omega_2 \\ \Im\omega_2\end{array}\right]}
{\left[\begin{array}{crcr}\Re A & -\Im A &  &  \\ \Im A & \Re A&  &  \\ &  & \Re B &-\Im B  \\ &  & \Im B & \Re B \end{array}\right] 
\left[\begin{array}{c}\Re\omega_1 \\ \Im\omega_1 \\ \Re\omega_2 \\ \Im\omega_2\end{array}\right]
}_{\cH^4}.
$$
We say that $\Phi$ is {\it convex with respect to the pair $(A,B)$} if $H_\Phi^{(A,B)}[v;\omega]\geq 0$ for all $v,\omega$.

\subsection{Numerical range and sectoriality}
\label{s: njuska}
Given $\phi\in (0,\pi)$ define the sector
\begin{equation*}
\label{autobus karlovac-posedarje}
\bS_{\phi}=\{z\in\C\setminus\{0\}\,;\ |\arg z|<\phi\}.
\end{equation*}
Also set $\bS_{0}=(0,\infty)$. 
Suppose that $\oA$ is a closed densely defined linear operator on a complex Banach space $X$. We denote its {\it spectrum} by $\sigma(\oA)$. Let $\th\in[0,\pi)$. Following \cite{Haase}, we say that $\oA$ is {\it sectorial of angle $\th$} if:
\begin{itemize}
\item
$\sigma(\oA)\subseteq \overline\bS_\vartheta$ and
\item
for every $\e\in(0,\pi-\vartheta)$ we have
$$
\sup_{z\in\C\backslash\overline\bS_{\vartheta+\e}}|z|\cdot\nor{(\oA-zI)^{-1}}<\infty\,.
$$
\end{itemize} 
Operators which are sectorial of some angle from $[0,\pi)$ will simply be called {\it sectorial}. In such a case the number 
$$
\omega(\oA):=\inf\Mn{\vartheta\in[0,\pi/2)}{\oA \text{ is sectorial of angle }\vartheta }
$$
is called the {\it sectoriality angle} of $\oA$.

\medskip
If $\cH$ is a Hilbert space, $\sk{\cdot}{\cdot}_{\cH}$ the scalar product on $\cH$ and $T:\cD(T)\rightarrow \cH$ a densely defined linear operator on $\cH$, we denote by $\cW(T)$ the {\it numerical range} of $T$; that is,
$$
\cW(T)=\mn{\sk{T h}{h}_{\cH}}{h\in \cD(T),\ |h|=1}.
$$ 
When $\cW(T)\subset\overline\bS_{\beta}$ for some $\beta\in(0,\pi)$, we define 
$$
\nu(T):=\inf\Mn{\beta\in(0,\pi)}{\cW(T)\subset\overline\bS_\beta},
$$
that is, $\overline\bS_{\nu(T)}$ is the smallest closed sector which contains the numerical range of $T$.

Furthermore, following \cite{McI} we say that $T$ is {\it $\omega$-accretive} for some $\omega\in[0,\pi/2]$ 
if $\sigma(T)\cup\cW(T)\subset\overline\bS_\omega$. Hence $\nu(T)\leq\omega$ in this case. If $T$ is bounded then $\sigma(T)\subset\overline{\cW(T)}$, thus such $T$ are $\omega$-accretive {\sl precisely} when $\nu(T)\leq\omega$. Any $\omega$-accretive operator is sectorial of angle $\omega$. 

\medskip
Let us return to operators in divergence form. The two accretivity conditions \eqref{eq: elli} and \eqref{eq: ne mogu se kontrolirati}
imply that, for a.e. $x\in\Omega$, the matrix $A(x)$ is $\arccos(\lambda/\Lambda)$-accretive as an operator on the Hilbert space $\C^n$. 
Define
\begin{equation}
\label{eq: Don Giovanni}
\nu(A)=\underset{x\in\Omega}{{\rm ess}\sup}\ \nu(A(x)).
\end{equation}
Then $L_{A}$ is $\nu(A)$-accretive, see \cite{McI}, and $\omega(L_{A})\leq\nu(L_A)\leq\nu(A)\leq\arccos(\lambda/\Lambda)$.

\section{Outline of the proof of the bilinear embedding}
\label{s: sokolov hortus musicus}

Our approach towards Theorem~\ref{t: bilincomplex} consists of defining and studying the {\it heat flow} associated with a particular {\it Bellman function}. The key property of the flow is a quantitative estimate of its derivative \eqref{eq: heatflow p1}. Using integration by parts, we single out the parallel property of the Bellman function alone that implies \eqref{eq: heatflow p1}. It could be perceived as a variant of ``convexity'' associated with the pair of accretive matrices in question. An adequate function turns out to be one constructed by Nazarov and Treil \cite{NT} in 1995. Its properties are formulated in Theorem~\ref{t: bouga}. In proving it we use the fact that their function is composed of tensor products of power functions. This makes the analysis of generalized convexity of power functions an essential part of our proof.
See Section~\ref{s: garanča}.

A simpler version of the Bellman-heat method was also the way through which Theorem~\ref{t: oktyabr'} was proven in \cite{DV-Kato}. 
The other works which stimulated thoughts developed in this paper were \cite{CD-mult, CD-OU}.

\medskip
The Bellman function technique has become widely known in harmonic analysis since the mid 1990s, following the work by Nazarov, Treil and Volberg \cite{NTV}. Afterwards it has been employed in a large number of papers, of which the closest ones to our approach (that is, those where Bellman functions are explictly paired with heat flows) are \cite{DV-Kato, CD-mult, PV, NV, DV-AB, DV, DV-Sch, CD, DP, PSW, MS}.

For another perspective on heat-flow techniques, various examples and references we refer the reader to the papers by Bennett et al. \cite{B, BCCT}.

\subsection{The heat-flow method expanded}
\label{s: heat-flow}
In this section we illustrate in more detail the heat-flow technique we will utilize for proving the bilinear embedding in Theorem~\ref{t: bilincomplex}. The exposition will be rather descriptive, aimed at giving the idea of the proof without dwelling on technical details which will be addressed later.

\smallskip
When proving the bilinear embedding of Theorem~\ref{t: bilincomplex}, a regularization argument (see the appendix) allows us to assume that the coefficients of the matrix functions $A,B\in\cA(\R^{n})$ are {\sl smooth}, that is, of class $C_b^1(\R^n)$ consisting of all bounded $C^1$ functions with bounded derivatives. 

Fix two test functions $f,g\in C_c^\infty(\R^n)$ and $\Phi:\C^2\rightarrow \R_+$ of class $C^1$. 
Suppose that $\psi\in C^\infty_c(\R^n)$ is a radial function, $\psi\equiv 1$ in the unit ball, $\psi\equiv 0$ outside the ball of radius $2$, and
$0< \psi< 1$ elsewhere. For $R>0$ define $\psi_R(x) := \psi(x/R)$. The choice of $A,B,f,g,\Phi,\psi,R$ gives rise to a function $\cE:[0,\infty)\rightarrow \R_+$ defined by 
$$
\cE(t)=\int_{\R^n} \psi_R \cdot \Phi\left(P_{t}^{A}f, P_{t}^{B}g\right).
$$

We say that the flow associated with $A,B$ and $\Phi$ is {\it regular} if, for every $f,g$ the function $\cE$ is continuous on $[0,\infty)$, continuously differentiable on $(0,\infty)$ and
\begin{equation*}
\label{turgenjev}
\cE'(t)=\int_{\R^n} \psi_R \cdot \frac{\pd}{\pd t}\Phi\left(P_{t}^{A}f,P_{t}^{B}g\right).
\end{equation*}

Fix $p>2$. We are interested in finding a function $\Phi\in C^1(\C^2)$, possibly depending on  $p$, such that for any $f,g$ the corresponding flow admits the following properties:
\begin{itemize}
\item
regularity;
\item
{\it quantitative monotonicity}, that is, 
the existence of $\text{\got a}_0=\text{\got a}_0(p,A,B)>0$ such that
\begin{equation}
\label{eq: heatflow p1}
-\cE'(t)
\geq \text{\got a}_0 
\sk{\psi_R|\nabla P_{t}^{A}f|}{|\nabla P_{t}^{B}g|}_{L^2(\R^n)}+(E.T.)\,,
\end{equation}
where $(E.T.)$ stands for ``error term'' which we expect to disappear as $R\rightarrow\infty$;
\item
initial value bound, that is, the existence of $\text{\got b}_0=\text{\got b}_0(p,A,B)>0$ such that
\begin{equation}
\label{eq: heatflow p2}
\cE(0)\leq \text{\got b}_0(\|f\|^p_p+\|g\|^q_q)\,.
\end{equation}
\end{itemize}
For if these conditions are fulfilled, then, for any $f,g$ as above,
$$
\aligned
\text{\got a}_0\int^\infty_0
\sk{\psi_R|\nabla P_{t}^{A}f|}{|\nabla P_{t}^{B}g|}_{L^2(\R^n)}&\,dt
+\int_0^\infty(E.T.)\\
& \leq -\int^{\infty}_{0}\cE'(t)\,dt\leq \cE(0)\leq \text{\got b}_0(\|f\|^p_p+\|g\|^q_q).
\endaligned
$$

We would like to send $R\rightarrow\infty$. Since we are assuming that the coefficients of $A,B$ are smooth, $P_{t}^{A}$ is bounded on $L^{p}$ for any $1\leq p\leq \infty$ \cite[Theorem~4.8]{Auscher}, which enables us to show that $(E.T.)\rightarrow 0$ as $R\rightarrow\infty$. So we arrive at
$$
\text{\got a}_0\int^\infty_0
\sk{|\nabla P_{t}^{A}f|}{|\nabla P_{t}^{B}g|}_{L^2(\R^n)}\,dt
\leq \text{\got b}_0(\|f\|^p_p+\|g\|^q_q).
$$
By replacing $f$ with $\tau f$ and $g$ with $g/\tau$ and optimizing the right-hand side in $\tau>0$, we obtain the bilinear embedding \eqref{eq: bilincomplex},
$$
\int^\infty_0
\sk{|\nabla P_{t}^{A}f|}{|\nabla P_{t}^{B}g|}_{L^2(\R^n)}\,dt
\leqslant C(p,A,B)\|f\|_p\|g\|_q,
$$
for all $f\in C_c^\infty 
({\R^n})$ and $g\in C_c^\infty 
({\R^n})$, where $C(p,A,B)=p^{1/p}q^{1/q}\text{\got b}_0/\text{\got a}_0$.

\medskip
\noindent{\bf Reduction of \eqref{eq: heatflow p1} and \eqref{eq: heatflow p2} to the properties of $\Phi$.}
We would like to translate \eqref{eq: heatflow p1} and \eqref{eq: heatflow p2} into (pointwise) conditions on $\Phi$ alone. 

Clearly, \eqref{eq: heatflow p2} holds provided that $0\leq \Phi(\zeta,\eta)\leq \text{\got b}_0\left(|\zeta|^p+|\eta|^q\right)$, for all $\zeta,\eta\in\C$.

As for \eqref{eq: heatflow p1}, it will be proven in Section~\ref{s: moszkowski} through integration by parts that {\it when $\Phi$ is of class $C^2$ and $A,B$ are smooth}, the regularity of the flow holds  and implies 
\begin{equation}
\label{eq: kamulator kolo}
-\cE'(t)=\int_{\R^n}\psi_R\cdot H_\Phi^{(A,B)}[h_t;\nabla h_t]+(E.T.),
\end{equation}
with $h_t=(P_{t}^{A}f,P_{t}^{B}g)$ and $\nabla h_t=(\nabla P_{t}^{A}f,\nabla P_{t}^{B}g)$, while $H_\Phi^{(A,B)}$ is as in \eqref{eq: tetastevka}. As said before, we use the smoothness of $A,B$ to show that $\lim_{R\rightarrow\infty}(E.T.)=0$. 
Consequently, for \eqref{eq: heatflow p1} it will be sufficient to have the following pointwise inequality:

\medskip
{\it for a.e. $x\in\R^n$ we have }
\begin{equation*}
H_\Phi^{(A,B)}[v;\omega]
\geq \text{\got a}_0|\omega_1||\omega_2|, 
\hskip 30pt \forall\,v\in\C^2,\ \forall\,  \omega=(\omega_1,\omega_2)\in\C^{n}\times\C^{n}
\,.
\end{equation*}
\noindent
It turns out that a function $\Phi$ satisfying this property, as well as the above-specified size estimate, exists when $\Delta_{p}(A,B)>0$ or, equivalently, $|1-2/p|<\mu(A,B)$. 

\medskip
\noindent{\bf Summary.}
Given $p>2$ and $A,B\in\cA(\R^n)$ satisfying $\Delta_{p}(A,B)>0$ or, equivalently, $|1-2/p|<\mu(A,B)$, the proof of Theorem~\ref{t: bilincomplex} eventually reduces to finding a $C^2$ function $\Phi:\C^2\rightarrow\R$ such that: 
\begin{enumerate}[(i)]
\item
the corresponding flow is regular;
\item 
$0\leqslant \Phi(\zeta,\eta)\leqsim |\zeta|^p+|\eta|^q$
for all $(\zeta,\eta)\in \C^2$;
\item
\label{eq: certina}
$H_{\Phi}^{(A,B)}[v;\omega]\geqsim 
|\omega_1||\omega_2|$
for any $v\in\C^2$ and $\omega=(\omega_1,\omega_2)\in\C^{n}\times\C^{n}$.
\end{enumerate}
We can relax the condition $\Phi\in C^2(\C^2)$ by requiring that $\Phi$ be of class $C^1$ and almost everywhere twice differentiable with locally integrable second-order partial derivatives. Then we can consider the flow corresponding to a regularization of $\Phi$ by standard mollifiers (see Section~\ref{s: prehlada}).

\section{Chain rule}
\label{s: chakrulo}

For $w=w_{1}+iw_{2}\in\C$, introduce the complex derivatives
$$
\partial_{\bar{w}}=\frac{\partial_{w_{1}}+i\partial_{w_{2}}}{2},\quad \partial_{w}=\frac{\partial_{w_{1}}-i\partial_{w_{2}}}{2}.
$$ 
Let $\Phi:\C^{2}\rightarrow \R$ be of class $C^2$. Recall the notation introduced in Sections~\ref{s: grand sonata luganskij} and \ref{s: cajkovskij grand sonata}. Define
$$
\aligned
\partial_{\bar{\zeta}}\Phi & =\partial_{\bar{\zeta}}(\Phi\circ\cW_{2}^{-1})\circ\cW_{2}\\ 
\partial_{\bar{\eta}}\Phi & =\partial_{\bar{\eta}}(\Phi\circ\cW_{2}^{-1})\circ\cW_{2}.
\endaligned
$$
\begin{lem}
\label{l: chain Phi}
Suppose that $\Phi$ is as above. Let $\varphi,\psi\in H^{1}(\R^{n})\cap L^{\infty}(\R^{n})$. Then $(\partial_{\bar\zeta}\Phi)\circ(\f,\psi)$ and  $(\partial_{\bar\eta}\Phi)\circ(\f,\psi)$ belong to $H^{1}_{{\rm loc}}(\R^{n})$,  $\nabla\left((\partial_{\bar\zeta}\Phi)\circ(\f,\psi)\right)$ and $\nabla\left((\partial_{\bar\eta}\Phi)\circ(\f,\psi)\right)$ belong to $L^{2}(\R^{n};\C^{2n})$, and 
$$
2 
\cW_{2n}\left( 
\nabla
   \left[(\partial_{\bar\zeta}\Phi)\circ(\f,\psi)\right],
   \nabla\left[(\partial_{\bar\eta}\Phi)\circ(\f,\psi)
\right]\right)
=
  \left[
     {\rm Hess}_{\cW_2}\left(\Phi;(\varphi,\psi)\right)\otimes I_{\R^n}
  \right]
  \cW_{2n}(\nabla\varphi,\nabla\psi).
$$  
\end{lem}

\begin{proof}
Write $\nabla$ for the gradient with respect to $x\in\R^{n}$ and $\overline\nabla$ for the gradient with respect to $(\zeta_{1},\zeta_{2},\eta_{1},\eta_{2})\in\R^{4}$. Let also 
$
\Psi:=\Phi\circ\cW_2^{-1}:\R^4\rightarrow\R
$
and
$
k:=\cW_2(\f,\psi):\R^n\rightarrow\R^4.
$
Then by the chain rule for weak derivatives \cite[Theorem~2.1.11]{Z},
\begin{align}
\label{eq: meiji}
2& \cW_{2n}\left( 
\nabla
   \left[(\partial_{\bar\zeta}\Phi)\circ(\f,\psi)\right],
   \nabla\left[(\partial_{\bar\eta}\Phi)\circ(\f,\psi)
\right]\right)\nonumber \\   
&=
\left(
\nabla
     \left( 
         \partial_{\zeta_1}\Psi\circ k
     \right),
\nabla
     \left( 
         \partial_{\zeta_2}\Psi\circ k
     \right),
\nabla
     \left( 
         \partial_{\eta_1}\Psi\circ k
     \right),
\nabla
     \left( 
         \partial_{\eta_2}\Psi\circ k
     \right)\right)\\
&=
\left(
   \left[
      \overline\nabla(\partial_{\zeta_{1}}\Psi)\circ k
   \right]
      \cdot\nabla k,
   \left[
      \overline\nabla(\partial_{\zeta_{2}}\Psi)\circ k
   \right]
      \cdot\nabla k,
   \left[
      \overline\nabla(\partial_{\eta_{1}}\Psi)\circ k
   \right]
      \cdot\nabla k,
   \left[
      \overline\nabla(\partial_{\eta_{2}}\Psi)\circ k
   \right]
      \cdot\nabla k
\right).\nonumber 
\end{align}
Recall that $k=(\Re\f,\Im\f,\Re\psi,\Im\psi)$ and observe that 
\begin{equation}
\label{cika smaje}
\nabla k=\nabla \cW_{2}(\f,\psi)=\cW_{2n}(\nabla\f,\nabla\psi)=(\nabla\Re\f,\nabla\Im\f,\nabla\Re\psi,\nabla\Im\psi).
\end{equation}
To ensure there is no ambiguity in \eqref{eq: meiji}, let us specify that for $G:\R^4\rightarrow \R$ we mean
$$
\aligned
\left[\overline\nabla\right.&\left. G\circ k\right]\cdot\nabla k\\
&= \left(\pd_{\zeta_1}G\circ k\right)\nabla\Re\f
 +\left(\pd_{\zeta_2}G\circ k\right)\nabla\Im\f
 +\left(\pd_{\eta_1}G\circ k\right)\nabla\Re\psi
 +\left(\pd_{\eta_2}G\circ k\right)\nabla\Im\psi.
\endaligned
$$
Putting all this together we see that \eqref{eq: meiji} equals
$$
\left[
     {\rm Hess}\left(\Psi;k\right)\otimes I_{\R^n}
\right]
\nabla k
=
\left[
     {\rm Hess}\left(\Phi\circ\cW_2^{-1};\cW_{2}(\varphi,\psi)\right)\otimes I_{\R^n}
\right]
\cW_{2n}(\nabla\varphi,\nabla\psi),
$$ 
just as claimed.
\end{proof}

\begin{cor}
\label{c: pavketov stakato}
Under the assumptions of Lemma~\ref{l: chain Phi}, for every $A,B\in\C^{n, n}$,
$$
\aligned
2\,\Re\sk{A\nabla\varphi}{\nabla\left[(\partial_{\bar\zeta}\Phi)(\varphi,\psi)\right]}_{\C^{n}}
+2\,\Re&\sk{B\nabla\psi}{\nabla\left[(\partial_{\bar\eta}\Phi)(\varphi,\psi)\right]}_{\C^{n}}\\
&\hskip 60pt
=H^{(A,B)}_{\Phi}\left[(\varphi,\psi);(\nabla\varphi,\nabla\psi)\right].
\endaligned
$$
\end{cor}
\begin{proof}
Write $h=(\varphi,\psi)$. Then $\nabla h=(\nabla\varphi,\nabla\psi)$ and by \eqref{eq: profana}, \eqref{cika smaje} and Lemma~\ref{l: chain Phi},
$$
\aligned
 2\,\Re\sk{A\nabla\varphi}{\nabla\left[(\partial_{\bar\zeta}\Phi)\circ h \right]}_{\C^{n}}
\hskip -120pt 
& \hskip 120pt
+2\,\Re\sk{B\nabla\psi}{\nabla\left[(\partial_{\bar\eta}\Phi)\circ h \right]}_{\C^{n}}\\
&
=2\sk{\cM(A)\cV_{n}\left(\nabla\f\right)}{\cV_{n}\left(\nabla\left[(\partial_{\bar\zeta}\Phi)\circ h\right]\right)}_{\R^{2n}}\\ 
& \hskip 20pt
+2\sk{\cM(B)\cV_{n}\left(\nabla\psi\right)}{\cV_{n}\left(\nabla\left[(\partial_{\bar\zeta}\Phi)\circ h\right]\right)}_{\R^{2n}}\\
&=
\sk{
  \left[\cM(A)\oplus\cM(B)\right]  \cW_{2n} (\nabla h)
}{
2 
\cW_{2n}\left( 
\nabla
   \left[(\partial_{\bar\zeta}\Phi)\circ h\right],
   \nabla\left[(\partial_{\bar\eta}\Phi)\circ h
\right]\right)
}_{\R^{4n}}\\
&=
\sk{
  \left[\cM(A)\oplus\cM(B)\right] \cW_{2n}  \left(\nabla h\right)
}{
\left[
     {\rm Hess}_{\cW_2}\left(\Phi; h\right)\otimes I_{\R^n}
  \right]
  \cW_{2n}(\nabla h)
}_{\R^{4n}}.
\endaligned
$$
Now  the corollary follows from \eqref{eq: tetastevka}.
\end{proof}

\subsection{Integration by parts}
\label{s: moszkowski}
Here we prove the identity \eqref{eq: kamulator kolo}, which reduces the estimate \eqref{eq: heatflow p1} to the estimate (iii) of the function $\Phi$ itself. Let $A,B\in\cA(\R^{n})$ be matrix functions with coefficients of class $C_b^1(\R^n)$, and  let $\Phi:\C^2\rightarrow \R$ be a $C^2$ function. Fix $f,g\in C^{\infty}_{c}(\R^{n})$ and a real-valued $\psi\in C^{\infty}_{c}(\R^n)$. The analyticity of the semigroups $(P^{A}_{t})_{t>0}$ and $(P^{B}_{t})_{t>0}$ on $L^{2}(\R^{n})$ together with a theorem of Auscher \cite[Theorem~4.8]{Auscher} imply that $P^{A}_{t}f$ and $P^{B}_{t}g$ belong to $H^{1}(\R^{n})\cap L^{\infty}(\R^{n})$, for all $t>0$. From this is not hard to see that the flow $t\mapsto\int\psi\cdot\Phi(h_t)$ is regular, where $h_t=(P_{t}^{A} f, P_{t}^{B} g):\R^n\rightarrow \C^2$. Let us write $\nabla h_{t}=(\nabla P^{A}_{t}f,\nabla P^{B}_{t}g)$.

\begin{prop}
\label{p: razina vezut}
Let $A,B,f,g,\psi,\Phi,h_t$ be as above. Then 
$$
\aligned
-\frac d{dt}&\int_{\R^n}\psi\,\Phi
(h_t)
=\int_{\R^n}
\psi\cdot H_\Phi^{(A,B)}[h_t;\nabla h_t]\\
&\hskip 10pt+\int_{\R^n}2\,\Re\Big( \left[(\pd_{\bar\zeta}\Phi)(h_{t})\right]\cdot\sk{\nabla\psi}{A\nabla P_{t}^{A}f}_{\C^n}+\left[(\pd_{\bar\eta}\Phi)(h_{t})\right]\cdot\sk{\nabla\psi}{B\nabla P_{t}^{B}g}_{\C^n}\Big).
\endaligned 
$$
\end{prop}

The integral in the last line is the ``error term" $(E.T.)$ referred to in Section~\ref{s: heat-flow}.

\begin{proof}
From the regularity of the flow we get
\begin{equation*}
\label{eq: ostanes}
\aligned
-\frac d{dt}&\int_{\R^n}\psi\,\Phi(h_t)
 =-\int_{\R^n}\psi\,\frac \pd{\pd t}\Phi\left(P^{A}_{t}f,P^{B}_{t}g\right)
\\ 
& =2\,\Re\int_{\R^n}\psi\cdot \sk{(\pd_{\bar\zeta}\Phi)(h_{t})}{L_{A}P_{t}^{A}f}_\C
+2\,\Re\int_{\R^n}\psi\cdot \sk{(\pd_{\bar\eta}\Phi)(h_{t})}{L_{B}P_{t}^{B}g}_\C\,.
\endaligned
\end{equation*}
Recall that $P^{A}_{t}f,P^{B}_{t}g\in H^{1}(\R^{n})\cap L^{\infty}(\R^{n})$. Therefore, by Lemma~\ref{l: chain Phi}, the functions $\psi\cdot(\pd_{\bar\zeta}\Phi)(h_{t})$ and $\psi\cdot(\pd_{\bar\eta}\Phi)(h_{t})$ belong to $ H^{1}(\R^{n})$ and, for $\gamma=\zeta,\eta$,
$$
\aligned
\int_{\R^n}\psi\cdot & \sk{(\pd_{\bar\gamma}\Phi)(h_{t})}{L_{A}P^{A}_{t}f}_\C
 =\int_{\R^n} \sk{\nabla\left[\psi\cdot(\pd_{\bar\gamma}\Phi)(h_{t})\right]}{A\nabla P^{A}_{t}f}_{\C^n}\\
& =\left(\int_{\R^n} \psi\sk{\nabla[(\pd_{\bar\gamma}\Phi)(h_{t})]}{A\nabla P^{A}_{t}f}_{\C^n}
+\int_{\R^n} (\pd_{\bar\gamma}\Phi)(h_{t})\cdot\sk{\nabla\psi}{A\nabla P^{A}_{t}f}_{\C^n}\right)
\endaligned
$$
and similarly with $B$ in place of $A$. Now apply Corollary~\ref{c: pavketov stakato}. 
\end{proof}

As mentioned before, the point of the above proposition (and of this section) is that the ``quantitative monotonicity'' \eqref{eq: heatflow p1} reduces to suitable {\sl pointwise} estimates of the terms $H_\Phi^{(A,B)}[v;\omega]$ for any $v\in\C^2$ and $\omega\in
\C^n\times\C^n$. We will estimate $H_\Phi^{(A,B)}[v;\omega]$ in the case of a very particular $\Phi$, to which the next section is devoted.

\section{Power functions and the Bellman function of Nazarov and Treil}
\label{s: garanča}

Unless specified otherwise, we assume everywhere in this section that $p\geqslant 2$ and $q=p/(p-1)$. Let $\delta>0$. 
Introduce the function $\wp=\wp_{p,\delta}:\R_+\times\R_+\longrightarrow\R_+$ by
\begin{equation*}
\label{foreman}
\wp(u,v)=
u^p+v^{q}+\delta
\left\{
\aligned
& u^2v^{2-q} & ; & \ \ u^p\leqslant v^q\\
& \frac{2}{p}\,u^{p}+\left(\frac{2}{q}-1\right)v^{q}
& ; &\ \ u^p\geqslant v^q\,.
\endaligned\right.
\end{equation*}
The Bellman function we use is the function $ Q= Q_{p,\delta}:\C\times\C\longrightarrow\R_+$ defined by 
\begin{equation}
\label{eq: rupkina}
 Q(\zeta,\eta):=
\wp(|\zeta|,|\eta|)\,.
\end{equation}

The origins of $Q$ lie in the paper of F. Nazarov and S. Treil \cite{NT}. A modification of their function has been later applied by A. Volberg and the second author in \cite{DV,DV-Sch}. Here we use a simplified variant which comprises only two variables. It was introduced in \cite{DV-Kato} and used by the present authors in \cite{CD, CD-mult}. 

The construction of the original Nazarov--Treil function in \cite{NT} was one of the earliest examples of the so-called Bellman function technique, which had been systematically introduced in harmonic analysis shortly beforehand by Nazarov, Treil and Volberg \cite{NTV}.  
The name ``Bellman function'' stems from the stochastic optimal control, see \cite{NTV1} for details. The same paper \cite{NTV1} explains the connection between the Nazarov--Treil--Volberg approach and the earlier work of Burkholder on martingale inequalities, see \cite{Bu1} and also \cite{Bu2,Bu4}. If interested in the genesis of Bellman functions and the overview of the method, the reader is also referred to Volberg et al. \cite{NTV1,V,NT} and Wittwer \cite{W}. The method has seen a whole series of applications, yet until recently (see \cite{CD, CD-mult}) mostly in Euclidean harmonic analysis.

In the course of the last few years, the Nazarov--Treil function $Q$ was found to possess nontrivial properties that reach much beyond the need for which it had been originally constructed in \cite{NT}. These properties were used for proving several variants of the bilinear embedding. See \cite{DV-Sch, DV-Kato, CD-mult, CD-OU, MS}. In the present paper we continue the exploration of the properties of $ Q$ by proving that a sort of a generalized convexity may occur in the presence of arbitrary complex accretive matrices $A,B$ (Theorem~\ref{t: bouga}).

\medskip
It is a direct consequence of the above definition that the function $ Q$ belongs to $C^1(\C^2)$, and is of order $C^2$ everywhere {\it except} on the set
$$
\Upsilon=\mn{(\zeta,\eta)\in  \C\times\C}{(\eta=0)\vee (|\zeta|^p=|\eta|^q)}\,.
$$
The following estimates are also straightforward.

\begin{prop}
\label{p: 3}
For $(\zeta,\eta)\in\R^2\times\R^2$ we have
$
0\leqslant Q(\zeta,\eta)\leqslant (1+\delta)\left(|\zeta|^p+|\eta|^q\right)
$
and
\begin{align*}
&2|(\partial_{\zeta}Q)(\zeta,\eta)|\leqslant (p+2\delta)\max\{|\zeta|^{p-1},|\eta|\},\\
&2|(\partial_{\eta}Q)(\zeta,\eta)|\leqslant (q+(2-q)\delta)|\eta|^{q-1}.
\end{align*}
\end{prop}

Recall the notation from \eqref{eq: tetastevka}. We would like to estimate $H_ Q^{(A,B)}[v;\omega]$ from below. 
Since in this chapter we do not integrate, we can think of $A,B$ simply as {\it constant} accretive matrices. 
The desired estimate is formulated below and will be proven in Section~\ref{l: pejda}. It was instrumental for our proof of the bilinear embedding (Section \ref{s: sonno}) and, on the other hand, it strengthened our belief in $\Delta_p(A)$ and \eqref{eq: aferim jano}; see Remark \ref{r: tovar grlo stima} for explanation.

\begin{thm}
\label{t: bouga}
Let $p\geq2$. Suppose that $A,B\in\cA_{\lambda,\Lambda}(\Omega)$ satisfy $\Delta_p:=\Delta_p(A,B)>0$.
Then there exists $\delta=\delta(\Delta_p,\lambda,\Lambda)\in(0,1)$ such that for $ Q= Q_{p,\delta}$ as above we have, for almost every $x\in\Omega$,
\begin{equation}
\label{eq: schnee}
H_{ Q}^{(A(x),B(x))}[v;\omega]
\geq \frac{\Delta_p}5\cdot\frac\lambda\Lambda|\omega_1||\omega_2|\,,
\end{equation}
for any $v\in\C^2\setminus\Upsilon$ and  $\omega=(\omega_1,\omega_2)\in\C^{n}\times\C^{n}$.
\end{thm}

\begin{rem}
\label{r: stojkovic}
The present authors proved a convexity result \cite[Theorem~15]{CD-mult} for $Q$, which was vital for their obtaining the sharp version of the holomorphic functional calculus in sectors on $L^{p}$ for generators of symmetric contraction semigroups. We note that Theorem~\ref{t: bouga} is basically a generalization of \cite[Theorem~15]{CD-mult} from the special case $A=e^{i\phi}I$ and $B=A^*$ to the general case considered above. Compare with \eqref{eq: perasovic}.
\end{rem}

\subsection{Regularization of $Q$} 
\label{s: prehlada}
We would like to replace $Q$ by a function which satisfies the inequality \eqref{eq: schnee} of Theorem~\ref{t: bouga} but is, in addition, also of class $C^2$ everywhere on $\C^2$ (not only on $\C^2\backslash\Upsilon$). A standard way of achieving this involves mollifiers. 

Denote by $*$ the convolution in $\R^4$ and let $(\f_\kappa)_{\kappa>0}$ be a nonnegative, smooth and compactly supported approximation of the identity on $\R^4$. If $\Phi:\C^2\rightarrow \R$, define $\Phi*\f_\kappa=(\Phi_\cW*\f_\kappa)\circ \cW:\C^2\rightarrow \R$. Explicitly, for $\zeta,\eta\in\C$,
$$
(\Phi*\f_\kappa)(\zeta,\eta)=\int_{\R^2\times\R^2}\Phi_\cW\left(\cV(\zeta)-s,\cV(\eta)-t\right)\f_\kappa(s,t)\,ds\,dt.
$$
The next result follows from Proposition~\ref{p: 3}. 

\begin{cor}
\label{c: delighted}
For $\zeta,\eta\in\C$ and $\delta,\kappa\in(0,1)$ we have
$$
0\leqslant (Q*\f_\kappa)(\zeta,\eta)\leqslant (1+\delta)\left[(|\zeta|+\kappa)^p+(|\eta|+\kappa)^q\right]
$$
and 
\begin{equation}
\label{eq: 442'}
\tag{{\rm ii'}}
\aligned
& \left|\partial_\zeta (Q*\f_\kappa)(\zeta,\eta)\right|\leqslant (p+2\delta)\max\left\{(|\zeta|+\kappa)^{p-1}, |\eta|+\kappa\right\}\\ 
& \left|\partial_\eta (Q*\f_\kappa)(\zeta,\eta)\right|\leqslant (q+(2-q))(|\eta|+\kappa)^{q-1}\,.
\endaligned
\end{equation}
\end{cor}
The following result is equivalent to the fact that the inequality in Theorem~\ref{t: bouga} is valid also in the distributional sense.

\begin{cor}
\label{c: Beethoven9}
Let $p\geq2$. Suppose that $A,B\in\cA_{\lambda,\Lambda}(\Omega)$ satisfy $\Delta_p:=\Delta_p(A,B)>0$.
Then there exists $\delta=\delta(\Delta_p,\lambda,\Lambda)\in(0,1)$ such that for $ Q= Q_{p,\delta}$ as above and $\kappa>0$
we have, for almost every $x\in\Omega$, 
\begin{equation}
\label{eq: wut}
H_{Q*\f_\kappa}^{(A(x),B(x))}[v;\omega]
\geq \frac{\Delta_p}5\cdot\frac\lambda\Lambda|\omega_1||\omega_2|\,,
\end{equation}
for any $v\in\C^2$, $\omega=(\omega_1,\omega_2)\in\C^{n}\times\C^{n}$.
\end{cor}

\begin{proof}
Since $ Q_\cW\in C^1(\R^4)$ and its second-order partial derivatives exist on $\R^4\backslash\cW(\Upsilon)$ and are locally integrable in $\R^4$, 
by \cite[Th\'eor\`eme~V, p. 57]{Schwartz}, see also \cite[Theorem~2.1]{Hebey}, we have that for any $v\in\C^2$, $\omega\in\C^{n}\times\C^{n}$ and $\kappa>0$,
\begin{equation*}
H_{ Q*\f_\kappa}^{(A,B)}[v;\omega]=\int_{\R^4}H_{ Q}^{(A,B)}[v-\cW^{-1}(\xi);\omega]\f_\kappa(\xi)\,d\xi\,.
\end{equation*}
Now Theorem \ref{t: bouga} immediately implies \eqref{eq: wut}.
\end{proof}

\subsection{Power functions }

For $r>0$ define the {\it power function} (by which we actually mean powers of the {\it modulus}) 
$$
\begin{array}{rccl}
F_r: & \C & \longrightarrow & \R_+\\
& \zeta &\longmapsto & |\zeta|^r. 
\end{array}
$$
Let $\bena$ denote the constant function of value $1$ on $\C$, that is, $\bena=F_0$. Introduce the notation 
$$
\widehat r=1-
2/r\,.
$$
If $p>1$, then 
$
|\wh p|=\wh{p^*}\,.
$
Here we remind the reader that $p^*=\max\{p,q\}$, where $1/p+1/q=1$. We can rewrite \eqref{eq: rupkina} as
\begin{equation}
\label{eq: prazno}
\aligned
 Q & = [1+(1-\wh p)\delta] F_p\otimes\bena +(1+\wh p\delta)\bena\otimes F_q\,, & & \hskip 20pt {\rm if }\ 
|\zeta|^p\geqslant |\eta|^q\\
 Q & =F_p\otimes\bena+\bena\otimes F_q+\delta F_2\otimes F_{2-q}\,, & & \hskip 20pt {\rm if }\ 
|\zeta|^p\leqslant |\eta|^q
\,,
\endaligned
\end{equation} 
which brings us to considering $H_\Phi^{(A,B)}[v;\omega]$ with $\Phi$ of the form $F_r\otimes F_s$ for some $r,s\geq 0$. 
For the sake of transparency we consider the two relevant cases separately:
\begin{enumerate}[(a)]
\item
$\Phi=F_p\otimes\bena$
or $\Phi=\bena\otimes F_q$;
\item
\label{c}
$\Phi=F_2\otimes F_{2-q}$.
\end{enumerate}

Given $\psi\in\R$ define
$$
\cK(\psi):=
    \left[
      \begin{array}{rr}
        \cos\psi  &  \sin\psi\\
        \sin\psi  & -\cos\psi
      \end{array}
    \right]
=
    \left[
      \begin{array}{rr}
        \cos\psi  & -\sin\psi\\
        \sin\psi  &  \cos\psi
      \end{array}
    \right]
    \left[
      \begin{array}{rr}
        1  &  0\\
        0  & -1
      \end{array}
    \right]
\,.
$$
The relevance of $\cK(\psi)$ for us stems from the formula
\begin{equation}
\label{eq: ella sunshine}
{\rm Hess}_\cV(F_r;\zeta)=\frac{r^2}{2}|\zeta|^{r-2}\big(I_2+\wh r\,\cK(2\arg\zeta)\big),
\end{equation}
valid for $r>0$ and $\zeta\in\C\backslash\{0\}$.
This formula is used in the proof of the next result that explains the emergence of the operator $\cI_p$ from 
\eqref{eq: nindze za 30kn}.
First extend $\cI_p$ to $p>0$ by the same rule 
\eqref{eq: nindze za 30kn}.

\begin{lem}
\label{l: tatac}
Let $r>0$, $A\in\C^{n,n}$, $\zeta\in\C\backslash\{0\}$ and $\xi\in\C^{n}$. Then
\begin{equation}
\label{eq: nikkor 24mm 2.8}
H_{F_r}^A[\zeta;\xi]=
\frac{r^2}{2}|\zeta|^{r-2}\Re\left(\sk{A\xi}{\xi}_{\C^n}+\wh r e^{-2i\arg\zeta}\sk{A\xi}{\bar\xi}_{\C^n}\right).
\end{equation}
Consequently, if $\varrho>0$ and $s\in\R$ then
\begin{equation}
\label{eq: svi se punti ovdi zbroje}
H_{F_r}^A[\varrho e^{is};\xi]=\frac{r^2}2\varrho^{r-2}\,\Re\!\sk{A\left(e^{-is}\xi\right)}{\cI_r\left(e^{-is}\xi\right)}_{\C^n}.
\end{equation}
\end{lem}

\begin{proof}
From \eqref{eq: ulica marata 5} and \eqref{eq: ella sunshine} we obtain
$$
H_{F_r}^A[\zeta;\xi]=
\frac{r^2}{2}|\zeta|^{r-2}
\big(
\sk{\cM(A)\cV(\xi)}{\cV(\xi)}_{\R^{2n}}
+\wh r\sk{\cM(A)\cV(\xi)}{\left(\cK(2\arg\zeta)\otimes I_{\R^n}\right)\cV(\xi)}_{\R^{2n}}
\big).
$$
Observe that, for $\psi\in\R$, one has $\cK(\psi)\otimes I_{\R^n}=\cM(e^{i\psi}I_n)U_n$, where 
$$
U_n
=\left[
\begin{array}{rr}
1  & \\
  & -1
\end{array}
\right]\otimes I_n\,.
$$
Notice also that $U_n\cV(\xi)=\cV(\bar\xi)$. Now \eqref{eq: profana} gives
$$
\aligned
\sk{\cM(A)\cV(\xi)}{\left(\cK(2\arg\zeta)\otimes I_{\R^n}\right)\cV(\xi)}_{\R^{2n}}
&=\sk{\cM(e^{-2i\arg\zeta}A)\cV(\xi)}{\cV(\bar\xi)}_{\R^{2n}}\\
&=\Re\!\sk{e^{-2i\arg\zeta}A\xi}{\bar\xi}_{\C^{n}}.
\endaligned
$$
This finishes the proof of \eqref{eq: nikkor 24mm 2.8}. 

The identity \eqref{eq: svi se punti ovdi zbroje} follows from \eqref{eq: nikkor 24mm 2.8} and 
\eqref{eq: nindze za 30kn}.
\end{proof}

We list a couple of straightforward consequences of \eqref{eq: nikkor 24mm 2.8}. 

\begin{lem}
\label{l: gdje carlija vjetric mio}
Suppose that $\zeta\in\C\backslash\{0\}$, $\xi\in\C^n$, $A\in\C^{n,n}$, $ t \in\R$ and 
$r>0$.
Then:
\begin{enumerate}[(1)]
\item
\label{l: dva konja na vodu}
$H_{F_2}^A[\zeta;\xi]=2\,\Re\sk{A\xi}{\xi}$\vskip 5pt
\item
\label{l: inspector}
$H_{F_r}^A[ t \zeta;\xi]=| t |^{r-2}H_{F_r}^A[\zeta;\xi]$\vskip 5pt
\item
\label{l: callahan}
$H_{F_r}^A[\zeta; t \xi]= t^2H_{F_r}^A[\zeta;\xi]$\vskip 5pt
\item
$H_{F_r}^A[i\zeta;\xi]=H_{F_r}^A[\zeta;i\xi]$\vskip 5pt
\item
\label{l: judge duvall}
$H_{F_r}^A[\zeta;\xi]=|\zeta|^{r-4}H_{F_r}^A[1;\bar\zeta\xi]$\vskip 5pt
\item
\label{l: have you ever seen the rain}
$H_{F_r}^A[\zeta;\xi]=H_{F_r}^{\bar A}[\bar\zeta;\bar\xi]$.
\end{enumerate}
\end{lem}

In view of Lemma~\ref{l: tatac} and Lemma~\ref{l: gdje carlija vjetric mio} {\it (\ref{l: judge duvall})}, it may be useful to observe the following. Write $A=U+iV\in\C^{n,n}$ and $\xi=\alpha+i\beta\in\C^n$. Recall the notation \eqref{eq: verve}. Then
\begin{equation}
\label{eq: ichimonji}
\aligned
\Re\sk{A\xi}{\xi}_{\C^n} & = \sk{U_{\sf s}\alpha}{\alpha}+\sk{U_{\sf s}\beta}{\beta}+2\sk{V_{\sf a}\alpha}{\beta}\\
\Re\!\sk{A\xi}{\bar\xi}_{\C^n} &=\sk{U_{\sf s}\alpha}{\alpha}-\sk{U_{\sf s}\beta}{\beta}-2\sk{V_{\sf s}\alpha}{\beta}.
\endaligned
\end{equation}

The next result reveals how $\Delta_p(A)$ arose from generalized Hessian forms of power functions. We extend $\Delta_p(A)$ to $p>0$ by the same definition \eqref{eq: kabuto}.

\begin{prop}
\label{p: kaminszki}
Let $A\in L^\infty(\Omega\rightarrow\C^{n,n})$ and $p\in(0,\infty)$. 
Then
\begin{equation}
\label{eq: dje se kupas}
\aligned
\Delta_p(A)
& =\frac2{p^2}\,
\underset{x\in\Omega}{{\rm ess}\inf}
\min_{|\xi|=1}
\min_{|\zeta|=1}
H_{F_p}^{A(x)}[\zeta;\xi]\,.
\endaligned
\end{equation}
If $p\geq1$ then also 
$\Delta_p(A)=\Delta_q(A)$, where $q=p/(p-1)$ if $p\ne1$ and $q=\infty$ if $p=1$. 

\end{prop}

\begin{proof}
From Lemma~\ref{l: gdje carlija vjetric mio} 
{\it (\ref{l: judge duvall})} we quickly get
\begin{equation}
\label{eq: kad si stigla}
\min_{|\xi|=1}
\min_{|\zeta|=1}
H_{F_p}^{A(x)}[\zeta;\xi]
=\min_{|\eta|=1}
H_{F_p}^{A(x)}[1;\eta].
\end{equation}
As a special case of \eqref{eq: svi se punti ovdi zbroje} we have
\begin{equation}
\label{eq: i do kad si tu}
H_{F_p}^{A(x)}[1;\eta]=
\frac{p^2}2\,\Re\sk{A(x)\eta}{\cI_p\eta}_{\C^n}.
\end{equation}
Combining \eqref{eq: kad si stigla} with \eqref{eq: i do kad si tu} and taking the essential infimum in $x\in\Omega$ proves \eqref{eq: dje se kupas}.

In order to show that $\Delta_p(A)=\Delta_q(A)$ for $p\geq1$, one may either notice the connection $\cI_q(i\eta)=i\cI_p(\eta)$ and use it in the definition of $\Delta_q(A)$, or else 
deduce from \eqref{eq: nikkor 24mm 2.8} that
\begin{equation}
\label{eq: nije presa}
\frac{2}{p^2}\min_{|\zeta|=1}
H_{F_p}^{A(x)}[\zeta;\xi]=
\Re\sk{A(x)\xi}{\xi}-|1-2/p|\cdot\left|\sk{A(x)\xi}{\bar\xi}\right|
\end{equation}
and then use $|1-2/p|=|1-2/q|$ together with \eqref{eq: dje se kupas}.
\end{proof}

\begin{rem}
\label{r: tovar grlo stima}
Ignoring the normalizing factor $2/p^2$, the formula \eqref{eq: dje se kupas} was how we initially defined $\Delta_p(A)$. Let us 
say a few words about the origin of this definition. 

As explained in Section \ref{s: sokolov hortus musicus}, our efforts to prove the bilinear estimate reduced to finding a (Bellman) function $Q=Q(\zeta,\eta)$ which is convex in a generalized sense with respect to the pair of matrices $(A,B)$, see Section \ref{s: cajkovskij grand sonata} for definitions. Our prime candidate was  the Nazarov--Treil function $Q$. In very particular cases of $A$ studied in \cite{CD-mult}, see Remark \ref{r: stojkovic}, and \cite{CD-OU}, we showed that the convexity of $Q$ with respect to $(A,A^*)$ reduces to the convexity of $A$ with respect to the building blocks of $Q$ - power functions $F_p$. In the present work we tried to find the adequate formulation and proof of this principle for pairs of general complex elliptic matrix functions $(A,B)$. Eventually, our answers to this question evolved into the definition \eqref{eq: dje se kupas}, the condition \eqref{eq: aferim jano} and Theorem \ref{t: bouga}. One may thus, with some reservation, view power functions as {\it Bellman functions in one variable}.
\end{rem}

Lemma~\ref{l: gdje carlija vjetric mio} {\it (\ref{l: inspector}, \ref{l: callahan})} and Proposition~\ref{p: kaminszki} immediately give the following estimate:

\begin{cor}
\label{c: cheverny}
Let $r>0$, $A\in\cA(\Omega)$, $\zeta\in\C\backslash\{0\}$ and $\xi\in\C^{n}$. 
Then a.e. $x\in\Omega$ we have
$$
H_{F_r}^{A(x)}[\zeta;\xi]\geq 
\frac{r^2}{2}|\zeta|^{r-2}|\xi|^2\Delta_r(A)\,.
$$
\end{cor}

\begin{lem}
\label{l: vsenoschnoe}
Let $1<q< 2$ and $A,B\in\cA(\Omega)$. Take $v=(\zeta,\eta)\in\C^2$ such that
$|\zeta|<|\eta|^{q-1}$ and $\omega=(\omega_1,\omega_2)\in\C^{n}\times\C^{n}$. Then
\begin{eqnarray}
\label{eq: bdenie}
 H_{F_2\otimes F_{2-q}}^{(A,B)}[v;\omega]
=F_{2-q}(\eta)H_{F_2}^A[\zeta;\omega_1]+F_{2}(\zeta)H_{F_{2-q}}^B[\eta;\omega_2]\hskip 110pt 
\\
+
2(2-q) 
|\eta|^{-q}
\sk{\left[\left(\cV(\zeta)\cdot\cV(\eta)^T\right)\otimes I_{\R^n}\right]\cV(\omega_2)}{\cM(A)\cV(\omega_1)}_{\R^{2n}}\hskip 5.4pt\nonumber\\
+
2(2-q) 
|\eta|^{-q}
\sk{\left[\left(\cV(\eta)\cdot\cV(\zeta)^T\right)\otimes I_{\R^n}\right]\cV(\omega_1)}{\cM(B)\cV(\omega_2)}_{\R^{2n}}
\,.\nonumber
\end{eqnarray}
\end{lem}

\begin{proof}
Combine the definition of $H^{(A,B)}_{F_{2}\otimes F_{2-q}}[v;\omega]$, see \eqref{eq: tetastevka}, and the identity
\[
\pd^2_{\zeta_j\eta_k}(F_{2}\otimes F_{2-q})(\zeta,\eta)=2(2-q)\zeta_j\eta_k|\eta|^{-q}, 
\hskip 30pt \text{for } j,k=1,2.
\qedhere
\]
\end{proof}

\begin{cor}
\label{c: omsk}
Let $1<q< 2$ and $A,B\in\cA_{\lambda,\Lambda}(\Omega)$. Take $v=(\zeta,\eta)\in\C^2$ such that
$|\zeta|<|\eta|^{q-1}$ and $\omega=(\omega_1,\omega_2)\in\C^{n}\times\C^{n}$.
Then for almost every $x\in\Omega$ we have
\label{slim shady}
$$
H_{F_2\otimes F_{2-q}}^{(A(x),B(x))}[v;\omega]
\geqslant 
2\lambda_A  |\eta|^{2-q}|\omega_1|^2
-4(2-q)\Lambda |\omega_1||\omega_2|
+\frac{(2-q)^2}2\Delta_{2-q}(B)|\eta|^{q-2}|\omega_2|^2.
$$
\end{cor}

\begin{proof}
Apply Lemma~\ref{l: vsenoschnoe}. In order to estimate the first two terms in \eqref{eq: bdenie} use Corollary~\ref{c: cheverny} with $r=2$ and $r=2-q$, while for the last two just note that
\begin{equation*}
\label{kolomna}
\aligned
\left|
  \left[
    \left(
       \cV(\zeta)\cdot\cV(\eta)^T
    \right)\otimes I_{\R^n}
  \right]
 \omega_2
\right|
& \leqslant 
|\zeta|\,|\eta|\,|\omega_2|\\
\left|
  \left[
    \left(
       \cV(\eta)\cdot\cV(\zeta)^T
    \right)\otimes I_{\R^n}
  \right]
 \omega_1
\right|
& \leqslant 
|\zeta|\,|\eta|\,|\omega_1|
\,.
\endaligned
\qedhere
\end{equation*}
\end{proof}

\subsection{More on $\Delta_p(A)$}
\label{s: power refresh}
Consider a matrix function $A:\Omega\rightarrow\C^{n,n}$. Write $A=U+iV$ for some real matrices $U,V$. Recalling the notation \eqref{eq: verve}, suppose that $U_{\sf s}(x)$ is positive definite (a.e. $x\in\Omega$). 
Observe that this condition is fulfilled for any $A\in\cA(\Omega)$. Denote by $S$ the operator $U_{\sf s}^{1/2}$.
For any $p>1$ define
\begin{equation}
\label{eq: the conspirators}
\oV_p
=\oV_p(V):=\frac{\sqrt{p-1}\,V-\sqrt{q-1}\,V^T}2
=\frac{p-2}{2\sqrt{p-1}}\,V_{\sf s}+\frac p{2\sqrt{p-1}}\,V_{\sf a}
\end{equation}
and 
$$
\oW_p=\oW_p(A):=S^{-1}\oV_p(V)S^{-1}\,.
$$
Notice that $\oV_2(V)=V_{\sf a}$.

The reason for introducing $\oV_p$ and $\oW_p$ was the next equivalence. For $M\in\R^{n,n}$ set 
$$
\nor{M}:=\max\Mn{|Mu|}{u\in\R^n,|u|=1}.
$$

\begin{prop}
\label{p: sahbaz}
Let $A$ be as above and $p>1$. The following statements are equivalent:
\begin{enumerate}[(1)]
\item
$\Delta_p(A)\geq0$;
\item
a.e. $x\in\Omega$: 
\hskip20pt
$
\sk{U(x)\alpha}{\alpha}+
\sk{U(x)\beta}{\beta}
+2\sk{\oV_p(V)(x)\alpha}{\beta}\geq0
\hskip10pt
\forall\alpha,\beta\in\R^n;
$ 
\item
a.e. $x\in\Omega$: 
\hskip20pt
$\nor{\oW_p(A)(x)}\leq1$.
\end{enumerate}
\end{prop}

\begin{rem}
\label{r: pletnev schumann op. 17}
Condition {\it (2)} above appears in \cite[(2.25)]{CM} after normalization.
\end{rem}

\begin{proof} 
First consider the case when $A$ is a constant matrix. 
Recalling \eqref{eq: Till I collapse} and \eqref{eq: profana}, we see that for any $c\in\R$, the condition $\Delta_p(A)\geq 2c$ can be expressed as
$$
(p-1)\sk{U\alpha}{\alpha}
+ \sk{U\beta}{\beta}
-(p-1)\sk{V\beta}{\alpha}
+ \sk{V\alpha}{\beta}
\geq pc(|\alpha|^2+|\beta|^2)
\hskip 30pt
\forall\alpha,\beta\in\R^n.
$$
By replacing $\alpha$ with $\alpha/\sqrt{p-1}$ and $\beta$ with $-\beta$ we get
$$
\sk{U\beta}{\beta}
+\sk{U\alpha}{\alpha}
+ 2\sk{\oV_p(V)\beta}{\alpha}
\geq c(p|\beta|^2+q|\alpha|^2)
\hskip 30pt
\forall\alpha,\beta\in\R^n.
$$
Notice that $\sk{U\alpha}{\alpha}=\sk{U_{\sf s}\alpha}{\alpha}=|S\alpha|^2$ and introduce $u=S\alpha$, $v=S\beta$.
Then the above inequality can be rephrased as
$$
|v|^2+|u|^2+
2\sk{\oW_p(A)v}{u}
\geq c\left(p|S^{-1}v|^2+q|S^{-1}u|^2\right).
$$
This should be valid for all $u,v\in\R^n$. In particular, we may replace $u$ by $-u$. Therefore $\Delta_p(A)\geq 2c$ is equivalent to
$$
|v|^2+|u|^2-2|\sk{\oW_p(A)v}{u}|\geq c\left(p|S^{-1}v|^2+q|S^{-1}u|^2\right),
\hskip 30pt
\forall u,v\in\R^n.
$$
From here it is not difficult to complete the proof in the constant case.

For the general, nonconstant case one uses that $\Delta_p(B)={{\rm ess}\inf}_{x\in\Omega}\Delta_p(B(x))$.
\end{proof}

\noindent
The proof of Proposition~\ref{p: sahbaz} also enables one to describe $\Delta_p(A)>0$ in similar terms.

We continue by an explicit comparison between $\Delta_p(A)$ and $\mu(A)$; recall that the latter was defined in \eqref{eq: potrosijo}.

\begin{prop}
\label{p: reci dje si}
Suppose that $A\in\cA(\Omega)$ and $p\in[1,\infty]$. Then $\Delta_p(A)=0$ if and only if $|1-2/p|=\mu(A)$. 
If the above equalities are not satisfied then 
\begin{equation}
\label{eq: admiral barovic}
\frac{\lambda_A}{\mu(A)}
\leq
\frac{\Delta_p(A)}{\mu(A)-|1-2/p|}
\leq
\Lambda_A\,.
\end{equation}
\end{prop}

\begin{proof}
The proof is based on the following formula which emerges from 
\eqref{eq: dje se kupas} and \eqref{eq: nije presa}:
\begin{equation}
\label{eq: pelikan}
\Delta_p(A)=\underset{x\in\Omega}{{\rm ess}\inf}\min_{|\xi|=1}
\left(\Re\sk{A(x)\xi}{\xi}-|1-2/p|\cdot\left|\sk{A(x)\xi}{\bar\xi}\right|\right).
\end{equation}
One factors out $\Re\!\sk{A(x)\xi}{\xi}$ or $\left|\sk{A(x)\xi}{\bar\xi}\right|$, depending on the part of \eqref{eq: admiral barovic} 
which is being proven, recalls \eqref{eq: elli}, \eqref{eq: ne mogu se kontrolirati} and \eqref{eq: potrosijo}, and applies the properties of essential infimum. The complete proof is rather elementary yet tedious, therefore we leave it out.
\end{proof}

Clearly, for $A=e^{i\phi}I_n$ (cf. Proposition~\ref{p: reki begalci}) we have equalities everywhere in \eqref{eq: admiral barovic}.

\begin{cor}
\label{c: niagara}
For any $A\in L^\infty(\Omega\rightarrow\C^{n,n})$, the function $p\mapsto\Delta_p(A)$ is Lipschitz continuous on $[1,\infty]$, increasing on $[1,2]$ and decreasing on $[2,\infty]$.
\end{cor}
\begin{proof}
The statement follows from \eqref{eq: pelikan}.
\end{proof}

\begin{cor}
\label{c: garrison}
Take any $A\in L^\infty(\Omega\rightarrow\C^{n,n})$ and $p\in[1,\infty]$. Then: 
\begin{enumerate}[1.)]
\item
\label{eq: ckali}
$\Delta_p(A)=\Delta_p(\bar A)$;
\item
$\Delta_p(A)=\Delta_p(QAQ^T)$ for every matrix function $Q:\Omega\rightarrow\cO(n)$, where $\cO(n)$ denotes the subset of $\R^{n,n}$ consisting of orthogonal matrices.
\end{enumerate}
If $\Delta_p(A)\geq0$ then 
\begin{enumerate}[1.)]
\addtocounter{enumi}{2}
\item
\label{eq: haeri}
${
\displaystyle
\Delta_p(A^*)\geq \frac{\Delta_p(A)}{p^*-1};
}$
\item
${
\displaystyle
\Delta_p(A_{\mathsf s})\geq\frac{\min\{p,q\}}2\,\Delta_p(A).
}$
\end{enumerate}
\end{cor}

\begin{proof}
It is enough to assume that $p\in(1,\infty)$. The statements for $p=1,\infty$ follow by continuity and monotonicity (Corollary \ref{c: niagara}).

\medskip
The first statement follows from Proposition~\ref{p: kaminszki} and Lemma~\ref{l: gdje carlija vjetric mio}{\it(\ref{l: have you ever seen the rain})}.

\medskip
The second statement is an easy consequence of the definition \eqref{eq: kabuto} and the fact that $\cI_p$ commutes with real matrices.

\medskip
Let us now address the third statement. 
Take $\xi\in\C^n$ with $|\xi|=1$, write $\eta=\cI_p\xi$ and observe that
$
\cI_q\cI_p=(4/pq)I_{\C^n}.
$
This implies $\xi=(pq/4)\cI_q\eta$ and thus for almost every $x\in\Omega$ we have
$$
\Re\sk{A(x)^*\xi}{\cI_p\xi}_{\C^n}
=\frac{pq}4|\eta|^2\Re\sk{\cI_q\frac\eta{|\eta|}}{A(x)\frac\eta{|\eta|}}_{\C^n}
\geq\frac{pq}4 |\cI_p\xi|^2\Delta_q(A).
$$
We know from Proposition \ref{p: kaminszki} that $\Delta_q(A)=\Delta_p(A)$. Since
${\displaystyle
\min\mn{|\cI_p\xi|}{|\xi|=1}=2/p^*,
}$
the assumption $\Delta_p(A)\geq0$ implies
$$
\Re\sk{A(x)^*\xi}{\cI_p\xi}
\geq 
\frac{pq}4\left(\frac2{p^*}\right)^2
\Delta_p(A)
=\frac{\Delta_p(A)}{p^*-1}.
$$
Finally minimize over $\xi$ and $x$.

\medskip
We now prove the last claim. One finds that the function $\Delta_p: L^\infty(\Omega\rightarrow\C^{n,n})\rightarrow\R$ is concave. Thus
$$
\Delta_p(A_{\mathsf s})\geq\frac{\Delta_p(A)+\Delta_p(A^T)}2=\frac{\Delta_p(A)+\Delta_p(A^*)}2
\geq\frac{1+(p^*-1)^{-1}}{2}\Delta_p(A).
$$
For the equality we used part {\it \ref{eq: ckali}.)}, while {\it \ref{eq: haeri}.)} gives the last inequality.
\end{proof}

The next result complements Proposition~\ref{p: sahbaz}.

\begin{prop}
\label{p: trazi se morricone}
Let $A$ be as in Proposition~\ref{p: sahbaz} and $p>1$. The following statements are equivalent:
\begin{enumerate}[(1)]
\item
$\Delta_p(A_{\sf s})\geq0$;

\item 
a.e. $x\in\Omega$: 
\hskip 20pt
$|p-2||\sk{V(x)\alpha}{\alpha}|\leq2\sqrt{p-1} \sk{U(x)\alpha}{\alpha}
\hskip 20pt\forall \alpha\in\R^n$;

\item
a.e. $x\in\Omega$: 
\hskip20pt
$\sk{A(x)\alpha}{\alpha}_{\C^n}\in\overline\bS_{\phi_p}$
\hskip 143.21pt $\forall\alpha\in\R^n$.
  
\end{enumerate}

\end{prop}

\begin{rem}
\label{hozier}
Condition {\it (2)} above appears in \cite[(5.23)]{CM}.
\end{rem}

\begin{proof}
Let us first prove that  $\textit{(1)}\Leftrightarrow\textit{(2)}$. Recall that $\sk{V_{\sf s}\alpha}{\alpha}=\sk{V\alpha}{\alpha}$ and the same for $U$. By \eqref{eq: the conspirators} and Proposition~\ref{p: sahbaz}, \textit{(1)} is equivalent to the following inequality, valid for almost every $x\in\Omega$:
\begin{equation}
\label{eq: square}
\hskip 20pt
|p-2||\sk{V_{\mathsf s}(x)\alpha}{\beta}|\leq\sqrt{p-1} \left(\sk{U(x)\alpha}{\alpha}
+\sk{U(x)\beta}{\beta}\right)
 \hskip 15pt\forall \alpha,\beta\in\R^n.
\end{equation}
This gives $\textit{(1)}\Rightarrow\textit{(2)}$. To prove $\textit{(2)}\Rightarrow\textit{(1)}$ write 
$4\sk{V_{\mathsf s}\alpha}{\beta}=\sk{V(\alpha+\beta)}{(\alpha+\beta)}-\sk{V(\alpha-\beta)}{(\alpha-\beta)}$. Estimating the right-hand side by \textit{(2)} proves \eqref{eq: square} and hence \textit{(1)}.

The equivalence $\textit{(2)}\Leftrightarrow\textit{(3)}$ follows from observing that $2\sqrt{p-1}/|p-2|=\tan\phi_p$. 
\end{proof}

\subsection{Examples. Connection with optimal results in the holomorphic functional calculus.}

\label{s: Howlin Wolf Hidden Charms}

Let us list a few cases of explicit identifications of $p$-ellipticity intervals which will be used in the continuation or else have appeared implicitly in our previous works.

The first result quickly follows from the definition \eqref{eq: kabuto}.

\begin{lem}
\label{l: adventni koledar 2018}
Let $\Omega\subset\R^{n}$ be open, $W\in L^\infty(\Omega\rightarrow\R^{n,n})$  antisymmetric and $\phi\in\R$. Define
$
A(x):=e^{i\phi}I_{\R^n}+iW(x)\,.
$
Then for any $p\in[1,\infty]$ we have
$$
\Delta_p(A)
=\cos\phi-\sqrt{(1-2/p)^2+\nor{W}^2}.
$$
Here 
$
{\displaystyle\nor{W}:=\underset{x\in\Omega}{{\rm ess}\sup}\nor{W(x)}}
$
with $\nor{W(x)}$ being the operator norm of $W(x)$ in $\cB(\R^n)$.
\end{lem}

It is worth recording the following special case of Lemma \ref{l: adventni koledar 2018}:
\begin{equation}
\label{eq: perasovic}
\Delta_p(e^{i\phi}I)=\cos\phi-|1-2/p|,
\hskip 40pt \forall\phi\in\R, \forall p\in[1,\infty].
\end{equation}

The threshold that we obtained in \cite[Lemma~20]{CD-mult} is by \eqref{eq: perasovic} equivalent to $\Delta_p(e^{i\phi}I)\geq0$. 
Let us illuminate this connection a little bit.

One of the novelties in \cite{CD-mult} was bilinear embedding with {\sl complex} time, expressed as the integration over the boundary of the sector $\bS_\phi$ in \cite[Theorem 9]{CD-mult}, that is, with real time $t$ replaced by $te^{\pm i\phi}$. Recall from Section \ref{s: heat-flow} that a part of our heat-flow argument which was also used in \cite{CD-mult} is differentiation of the flow with respect to $t$. This accounts for the appearance of factors $e^{\pm i\phi}$ attached to the second-order derivatives of the Bellman function $Q$ in \cite[Section 4]{CD-mult}. 
However, by using Theorem \ref{t: bouga} and the terminology introduced in this paper, positivity of those terms, which is a fundamental component of our heat-flow method, boils down to the $p$-ellipticity of matrices $e^{i\phi}I$. See also Remark~\ref{r: stojkovic}.

\medskip
We are able to calculate $\nor{\oW_p(A)}$ in a special case which appeared in our proof of the sharp bounded holomorphic functional calculus for nonsymmetric Ornstein--Uhlenbeck operators \cite{CD-OU}.

\begin{prop}
\label{eq: moji decki sa mnom u LJ}
Suppose that $B\in\R^{n,n}$ is such that $B_{\sf s}$ is positive definite. Then, for any $\phi\in(-\pi/2,\pi/2)$,
\begin{equation}
\label{eq: Wp in OU}
\nor{\oW_p\big(e^{i\phi}B\big)}^2
=\tan^2\phi\cdot
\frac{\Nor{B_{\sf s}^{-1/2}B_{\sf a}B_{\sf s}^{-1/2}}^2+{\widehat p}^2}{1-{\widehat p}^2}
\,.
\end{equation}
\end{prop}

\begin{proof}
The definition of $\oW_p$ in combination with \eqref{eq: the conspirators} yields
$$
\oW_p(e^{i\phi}B)
=\tan\phi\left(\frac{p-2}{2\sqrt{p-1}}\,I
+\frac p{2\sqrt{p-1}}\, 
B_{\sf s}^{-1/2}B_{\sf a}B_{\sf s}^{-1/2}\right).
$$
The fact that $B_{\sf s}^{-1/2}B_{\sf a}B_{\sf s}^{-1/2}=\oW_2\left(e^{i\pi/4}B\right)$
is antisymmetric implies \eqref{eq: Wp in OU}; see also \cite[Proof of Theorem~1.1]{CFMP2}.
\end{proof}

As a consequence we are able to determine when $\Delta_p\big(e^{i\phi}B\big)\geq0$. By Proposition~\ref{p: sahbaz} this happens precisely when $\nor{\oW_p\big(e^{i\phi}B\big)}\leq1$. Solving on $\phi\in[0,\pi/2)$ the equation $\nor{\oW_p\big(e^{i\phi}B\big)}=1$ gives the 
critical angle that featured in \cite{CD-OU}, see eq. (10) there.
\medskip

Another connection with \cite{CD-OU} is the next identity. It generalizes \cite[Proposition~21]{CD-OU}, where it was proven in the case of $\Re A=\Im A$. Indeed, this follows from applying Proposition~\ref{eq: moji decki sa mnom u LJ} with $\phi=\pi/4$.

\begin{prop}
For $A\in\C^{n,n}$ with $\Re A$ positive definite and $p\in(1,\infty)$ we have
\begin{equation}
\label{eq: smokestack lightnin'}
\sup_{\xi\in\C^n\backslash\{0\}\atop\zeta\in\C\backslash\{0\}}
\frac{\left|H_{F_p}^{i\Im A}[\zeta;\xi]\right|}{H_{F_p}^{\Re A}[\zeta;\xi]}
=
\nor{\oW_p(A)}.
\end{equation}
\end{prop}

\begin{proof}
By Lemma~\ref{l: gdje carlija vjetric mio} {\it (\ref{l: judge duvall})}, it is enough to take $\zeta=1$ in the supremum on the left. Write $A=U+iV$, as before. Denote the left-hand side of \eqref{eq: smokestack lightnin'} by $\gamma_p(A)$. By \eqref{eq: i do kad si tu}, $\gamma_p(A)$ is the smallest number for which the inequality
$$
\left|
\Im\!
\sk{
V\xi}{\cI_p\xi}
\right|
\leq\gamma_p(A)
\Re\!
\sk{
U\xi}{\cI_p\xi}
$$
is valid for all $\xi\in\C^n$. Rewrite the above inequality in terms of $u,v\in\R^n$, introduced through $\xi=U_{\mathsf s}^{-1/2}\left(u/\sqrt{p-1}+iv\right)$. Eventually we get
$$
\left|\sk{\oW_p(A)v}{u}\right|\leq\gamma_p(A)\,\frac{|v|^2+|u|^2}{2}.
$$
By polarization, this is of course equivalent to 
$$
\left|\sk{\oW_p(A)v}{u}\right|\leq\gamma_p(A)|v||u|
\hskip 40pt
\forall v,u\in\R^n,
$$
and the smallest $\gamma_p(A)$ in this inequality is by definition $\nor{\oW_p(A)}$.
\end{proof}

We leave the proof of the next result to the reader.
\begin{prop}
\label{p: ja sam}
If $A\in L^{\infty}(\Omega\rightarrow\C^{n,n})$ and $p\in[1,\infty]$ then for any pair of distinct numbers $\f,\psi\in(0,\pi/2)$ we have
$$
\Delta_p(A)\leq \left(\frac{\Delta_p(e^{i\f}A)}{\sin\f}-\frac{\Delta_p(e^{i\psi}A)}{\sin\psi}\right)/(\cot\f-\cot\psi).
$$
\end{prop}

It quickly follows from the definition \eqref{eq: kabuto} of $\Delta_p(A)$ that for $A,B\in L^{\infty}(\Omega\rightarrow\C^{n,n})$ and $1/p+1/q=1$ we have
$$
\left|\Delta_p(A)-\Delta_p(B)\right|\leq\frac{\nor{A-B}_\infty}{\min\{p,q\}}\,.
$$
In particular, $\f\mapsto\Delta_p(e^{i\f}A)$ is Lipschitz continuous on $(0,\pi/2)$. Hence, by Rademacher's theorem, this function is differentiable a.e. $(0,\pi/2)$. For $\f\in(0,\pi/2)$ at which the derivative exists, Proposition~\ref{p: ja sam} gives the estimate
$$
\frac\pd{\pd\f}\Delta_p(e^{i\f}A)\leq\frac{\Delta_p(e^{i\f}A)\cos\f-\Delta_p(A)}{\sin\f}\,.
$$

\subsection{Proof of Theorem~\ref{t: bouga}}
\label{l: pejda}
We will need the following straightforward statement which we formulate for the sake of convenience.

\begin{lem}
\label{tu tu tu po cesti}
Suppose that $a,b,c\in\R$. Then $\inf_{X>0}(a X-b+c X^{-1})>0$  if and only if 
$a,c\geq0$ and $b<2\sqrt{ac}$. In this case the above infimum equals $2\sqrt{ac}-b$.
\end{lem}

We partially follow the proof of \cite[Theorem~3]{DV-Sch}. 
When $p=2$ the Bellman function reads $Q(\zeta,\eta)=(1+\delta)|\zeta|^2+|\eta|^2$ for all $\zeta,\eta\in\C$, hence the theorem quickly follows from \eqref{eq: ella sunshine}. Thus from now on assume that $p>2$.

As in \cite{DV-Sch} write $\textsf{u}=|\zeta|$, $\textsf{v}=|\eta|$, $\textsf{ A}=|\omega_1|$, $\textsf{ B}=|\omega_2|$, 
where $v=(\zeta,\eta)\in\C^2\setminus\Upsilon$ and $\omega=(\omega_1,\omega_2)\in\C^n\times\C^n$. 
Following \eqref{eq: prazno} we consider two cases. 

\medskip
If $\textsf{u}^p> \textsf{v}^q>0$, then by \eqref{eq: prazno}, Corollary~\ref{c: cheverny} and Proposition~\ref{p: kaminszki} we have, almost everywhere on $x\in\Omega$, 
$$
\aligned
H_Q^{(A(x),B(x))}[v;\omega] 
&  =(1+(1-\wh p)\delta)
H_{F_p}^{A(x)}[\zeta;\omega_1]
+(1+\wh p\,\delta)
H_{F_q}^{B(x)}[\eta;\omega_2]\\
&\geqslant 
\frac{\Delta_p}2
\left[
p(p+2\delta) \textsf{u}^{p-2}\textsf{A}^2
+q(q+(2-q)\delta) \textsf{v}^{q-2}\textsf{B}^2
\right]\,.
\endaligned
$$
By the assumption we have $2-q>0$. So whenever $\delta>0$, we may continue as
$$
\geqslant 
\frac{\Delta_p}2
\left(
p^2 \textsf{u}^{p-2}\textsf{A}^2
+q^2 \textsf{v}^{q-2}\textsf{B}^2
\right)
\geq
\Delta_p pq\textsf{A}\textsf{B}\,.
$$
In the last step we used the inequality between the arithmetic and geometric mean and 
the assumption $\textsf{u}^p\geqslant \textsf{v}^q$.
 
\medskip
What remains is the case $\textsf{u}^p<\textsf{v}^q$. From \eqref{eq: prazno} and Corollaries \ref{c: cheverny} and \ref{c: omsk} we get, almost everywhere on $x\in\Omega$, 
\begin{equation}
\label{eq: his truth is marching on}
\aligned
H_Q^{(A(x),B(x))}[v;\omega] 
&  \geqslant
 H_{F_q}^{B(x)}[\eta;\omega_2] +\delta H_{F_2\otimes F_{2-q}}^{(A(x),B(x))}[v;\omega] 
 \\
& \geqslant 
2\delta\textsf{A}\textsf{B}
\left(
\lambda_A X
-2(2-q)\Lambda 
+\frac\Gamma4X^{-1}
\right),
\endaligned
\end{equation}
where $X=\textsf{v}^{2-q}\textsf{A}/\textsf{B}$ and 
\begin{equation*}
\label{eq: oher}
\Gamma=\frac{q^2\Delta_q(B)}\delta+(2-q)^2\Delta_{2-q}(B)\,.
\end{equation*}
We want 
$\lambda_A X-2(2-q)\Lambda +(\Gamma/4)X^{-1}\geqsim1$
uniformly in $X>0$. 
By Lemma~\ref{tu tu tu po cesti} this happens precisely when
\begin{equation}
\label{eq: prasko}
\frac{\Delta_q(B)}\delta>\left(\frac{2-q}{q}\right)^2\left(\frac{4\Lambda^2}{\lambda_A}-\Delta_{2-q}(B)\right)\,.
\end{equation}
From \eqref{eq: pelikan}, which holds also when $0<p<1$, we get the estimate
\begin{equation}
\label{angers}
\Delta_{2-q}(B)\geq\lambda-\Lambda q/(2-q),
\end{equation}
and one can eventually show that the condition \eqref{eq: prasko} is satisfied by taking 
\begin{equation}
\label{eq: schluss}
\delta=\frac{\lambda\Delta_q(B)}{10\Lambda^2}.
\end{equation}
In this case, again by Lemma~\ref{tu tu tu po cesti} and \eqref{angers}, we get, for any $X>0$,
$$
\aligned
\lambda_A X
-2(2-q)\Lambda 
+\frac\Gamma4X^{-1}
&\geq\sqrt{\lambda_A\Gamma}-2(2-q)\Lambda \\
&\geq\sqrt{10q^2\Lambda^2+(2-q)^2\lambda\left(\lambda-\Lambda q/(2-q)\right)}-2(2-q)\Lambda\\
&\geq\sqrt{10q^2\Lambda^2-q(2-q)\Lambda^2}-2(2-q)\Lambda\\
&\geq\Lambda\,.
\endaligned
$$
Remember from \eqref{eq: his truth is marching on} that in order to get an estimate of $H_Q^{(A,B)}[v;\omega]$ we need to multiply by $2\delta\textsf{A}\textsf{B}$. Estimate \eqref{eq: schnee} now follows. 

The theorem is proven with $\delta$ as in \eqref{eq: schluss}.
\qed

\section{Proof of the bilinear embedding (Theorem~\ref{t: bilincomplex})}
\label{s: sonno}

Take $p>1$, $n\in\N$ and $A,B\in\cA(\R^n)$ such that $\Delta_p(A,B)>0$.  
It is enough to consider the case $p\geq2$.
We will for the moment also {\it assume that $A,B\in C^1_b(\R^n)$}. 
Once the proof for smooth $A,B$ is over, we will apply the regularization argument from the Appendix to pass to the case of arbitrary (nonsmooth) $A,B$.

Let $\delta\in(0,1)$ be as in Theorem~\ref{t: bouga} (here we use the assumption that  $A$ and $B$ are $p$-elliptic) and let $Q= Q_{p,\delta}$ be the Bellman function defined in \eqref{eq: rupkina}. Take $f,g\in C_c^\infty(\R^n)$. Suppose that $\psi\in C^\infty_c(\R^n)$ is radial, $\psi\equiv 1$ in the unit ball, $\psi\equiv 0$ outside the ball of radius $2$, and $0< \psi< 1$ elsewhere. For $R>0$ define the dilates $\psi_R(x) := \psi(x/R)$. Let $(\f_\kappa)_{\kappa>0}$ be a nonnegative, smooth and compactly supported approximation of the identity on $\C^2$. 
Abbreviate $ Q*\f_\kappa= Q_\kappa$ and $h_t=(P_{t}^{A}f,P_{t}^{B}g)$. With these choices made and fixed, define for $t>0$ the quantity 
$\cE_{R,\kappa}$ by
\begin{equation*}
\label{eq: telavshi tushma gagpera}
\cE_{R,\kappa}(t)=\int_{\R^n}\psi_R\cdot Q_\kappa(h_t)\,.
\end{equation*}

As commented before, this flow is regular. Fix $T>0$. As indicated in Section~\ref{s: heat-flow}, we want to estimate the integral
\begin{equation}
\label{eq: Messiah}
-\int_0^T\cE_{R,\kappa}'(t)\,dt
\end{equation}
from above and below. 

\subsubsection*{Upper estimate of the integral \eqref{eq: Messiah}}
We have, by Corollary~\ref{c: delighted},
$$
-\int_0^T\cE_{R,\kappa}'(t)\,dt
\leqslant \cE_{R,\kappa}(0)
=\int_{\R^n}\psi_R\cdot Q_\kappa(f,g)
\leq(1+\delta)\int_{\R^n}\psi_R\left[(|f|+\kappa)^p+(|g|+\kappa)^q\right].
$$
By the Lebesgue dominated convergence theorem we may send first $\kappa\rightarrow0$ and then $R\rightarrow\infty$ and obtain
\begin{equation}
\label{eq: novica}
\limsup_{R\rightarrow\infty}\limsup_{\kappa\rightarrow0}\left(-\int_{0}^T\cE_{R,\kappa}'(t)\,dt\right)
\leqslant 
(1+\delta)(\nor{f}_p^p+\nor{g}_q^q)\,.
\end{equation}

\subsubsection*{Lower estimate of the integral \eqref{eq: Messiah}}

For $R>0$ define $\omega_R=\mn{x\in\R^n}{R\leqslant|x|\leq 2R}$, so that ${\rm supp}\ \nabla\psi_R\subset\omega_R$. 
Then, by Proposition~\ref{p: razina vezut} and Corollary \ref{c: Beethoven9},
$$
\aligned
-&\int_0^T\cE_{R,\kappa}'(t)\,dt \geqslant 
C(\Delta_p,\lambda/\Lambda)\int_0^T\int_{\R^n}\psi_R|\nabla P_{t}^{A}f||\nabla P_{t}^{B}g|\\
&+2\Re\int_0^T\int_{\omega_R}\Big( \left[(\pd_{\bar\zeta} Q_\kappa)\circ h\right]\cdot\sk{\nabla\psi_R}{A\nabla P_{t}^{A} f}_{\C^n}+\left[(\pd_{\bar\eta} Q_\kappa)\circ h\right]\cdot\sk{\nabla\psi_R}{B\nabla P_{t}^{B} g}_{\C^n}\Big).
\endaligned 
$$
Here $C(\Delta_p,\lambda/\Lambda)$ is the constant from \eqref{eq: wut}. 
We would like to study this inequality as $\kappa\rightarrow0$ and $R\rightarrow\infty$. Our argument is similar to the one from \cite[pp. 2825--2827]{DV-Kato} where real matrices were treated and thus the semigroup $(P_t^Af)_{t>0}$ was  $L^\infty$-contractive. In the complex case we slightly modify the argument: instead we use that the coefficients of $A$ are smooth and thus  $\left(P_t^A\right)_{t\in(0,T)}$ is uniformly bounded on $L^\infty$, as specified in \eqref{eq: Smarjeske Toplice} below.

\begin{lem}
\label{l: Mahalia Go Tell}
Let $A,B\in\cA(\R^n)\cap C_b^1(\R^n;\C^{n,n})$, $\alpha\geq0$ and $f,g\in C_c^\infty(\R^n)$. Then for all $R,T>0$ the map 
$(x,t)\mapsto \left|\nabla\psi_R(x)\right| \left|P_t^Af(x)\right|^\alpha\left|\nabla P_t^Bg(x)\right|$ belongs to $L^1\left(\R^n\times(0,T)\right)$ and
$$
\lim_{R\rightarrow\infty}
  \int_0^T
     \int_{\R^n}
        |\nabla\psi_R(x)| \left|P_t^Af(x)\right|^\alpha\left|\nabla P_t^Bg(x)\right|
     \,dx
   \,dt
=0.
$$
\end{lem}
\begin{proof}
By \cite[Theorem 4.8]{Auscher}, there exist $C,M>0$ such that 
\begin{equation}
\label{eq: Smarjeske Toplice}
\nor{P_t^Af}_\infty\leq C(1+t)^M\nor{f}_\infty, 
\hskip 20pt
\forall t>0.
\end{equation}
This implies, together with the identity $\nor{\nabla\psi_R}_\infty= \nor{\nabla\psi}_\infty/R$, the estimate
$$
\aligned
\int_0^T\int_{\R^n} |\nabla\psi_R| \left|P_t^Af\right|^\alpha\left|\nabla P_t^Bg\right| 
& \leqsim 
\frac1R\int_0^T\int_{\omega_R} 
\left|\nabla P_t^Bg\right|\\
& \leqsim 
R^{n/2-1}\int_0^T\Nor{\nabla P_t^Bg}_{L^2(\omega_R)}\,dt,
\endaligned
$$
with the implied constant depending on $\lambda_A,\Lambda_A,n,f,\alpha,T$.
If $R >0$ is large enough so that the open ball $B(0,R)$ in $\R^n$ contains the support of $g$, the $L^2$ off-diagonal estimates of Davies-Gaffney type \cite[Proposition~2.1]{A} imply
$$
\Nor{\nabla P_t^Bg}_{L^2(\omega_R)}\leqsim Ct^{-1/2}e^{-cR^2/t}\nor{g}_{L^2}
$$
for any $t>0$ and some $C,c>0$. From here we can quickly finish the proof.
\end{proof}

Since $ Q$ is of class $C^1$, we have 
$\pd_{\bar\zeta} Q_\kappa\rightarrow \pd_{\bar\zeta} Q$ and $\pd_{\bar\eta} Q_\kappa\rightarrow \pd_{\bar\eta} Q$ pointwise on $\R^n$, as $\kappa\rightarrow 0$. Therefore, by Corollary \ref{c: delighted} \eqref{eq: 442'}, the first part of Lemma \ref{l: Mahalia Go Tell} and the dominated convergence theorem, 
\begin{eqnarray*}
\label{eq: slava}
&\hskip -80pt
{\displaystyle\liminf_{\kappa\rightarrow 0}\left(-\int_0^T\cE_{R,\kappa}'(t)\,dt\right)
 \geqslant 
C(\Delta_p,\lambda/\Lambda)\int_0^T\int_{\R^n}\psi_R|\nabla P_{t}^{A}f||\nabla P_{t}^{B}g|}\\
&
{\displaystyle+2\Re\int_0^T\int_{\omega_R}\Big( (\pd_{\bar\zeta} Q)(h)\cdot\sk{\nabla\psi_R}{A\nabla P_{t}^{A} f}_{\C^n}
+(\pd_{\bar\eta} Q)(h)\cdot\sk{\nabla\psi_R}{B\nabla P_{t}^{B} g}_{\C^n}\Big).}
\nonumber
\end{eqnarray*}
Hence, by Corollary \ref{c: delighted} \eqref{eq: 442'}, the second part of Lemma \ref{l: Mahalia Go Tell} and Fatou's lemma, 
\begin{equation}
\label{eq: negovanovic}
\aligned
\liminf_{R\rightarrow \infty}\liminf_{\kappa\rightarrow 0}
\left(-\int_0^T\cE_{R,\kappa}'(t)\,dt\right)
& \geqslant 
C(\Delta_p,\lambda/\Lambda)\int_0^T\int_{\R^n}|\nabla P_{t}^{A}f||\nabla P_{t}^{B}g|.
\endaligned 
\end{equation}

\subsubsection*{Summary}
The combination of \eqref{eq: novica} and \eqref{eq: negovanovic} immediately gives rise to
$$
\int_0^T\int_{\R^n}|\nabla P_{t}^{A}f||\nabla P_{t}^{B}g|
\leq \frac{1+\delta}{C(\Delta_p,\lambda/\Lambda)}\,(\nor{f}_p^p+\nor{g}_q^q)\,.
$$
As usual, replace $f,g$ by $\tau f,\tau^{-1}g$ and minimize over $\tau>0$, after which send $T\rightarrow\infty$ and use the monotone convergence theorem. This gives the bilinear embedding \eqref{eq: bilincomplex} for smooth $A,B$. 

Finally, in order to treat the case of arbitrary $A$ and $B$, consider their mollifications $A_\e$ and $B_\e$ as in Section~\ref{s: BWV 1056 7. Allegro moderato}. Fix $t>0$. By Lemmas~\ref{l: stevens} and \ref{c: appregop}, $\nabla P^{A_\e}_{t}f$ converges to $\nabla P^{A}_{t}f$ 
in $L^{2}(\R^{n};\C^{n})$ as $\e\rightarrow 0$, and the same for $P_t^{B_{\e}}$. Therefore, 
$$
\int_{\R^{n}}\mod{\nabla P^{A}_{t}f}\mod{\nabla P^{B}_{t}g}
=
\lim_{\e\rightarrow 0}\int_{\R^{n}}\mod{\nabla P^{A_\e}_{t}f}\mod{\nabla P^{B_\e}_{t}g}.
$$
The conclusion now follows by integrating over $t$, applying the Fatou lemma, using the part proven so far (that is, the bilinear embedding for the smooth case) and Lemma~\ref{l: stevens}.
\qed

\subsection{Sharpness}
\label{s: zavratnica}

Proposition~\ref{p: reki begalci} follows from establishing the sharp angle of contractivity of the heat semigroup on $L^{p}$. 
Actually, the $L^{p}$ norm of the semigroup generated by the classical euclidean Laplacian on $\R^n$ at any complex time $z$ with $\Re z>0$ can be calculated explicitly:

\begin{thm}
\label{t: tgv bordeaux-cdg}
Suppose that $\phi\in(-\pi/2,\pi/2)$ and $p\in[1,\infty]$. Then there exists a constant $C=C(\phi,p)\geq1$ such that for all $n\in\N$ and $t>0$ we have
$$
\Nor{e^{-te^{i\phi}(-\Delta_n)}}_{\cB(L^{p}(\R^n))}= C^n\,.
$$
If
$|\phi|\leq\phi_p$ 
then $C=1$.\\
If $|\phi|>\phi_p$ then $C>1$. The constant $C$ can in this case also be given explicitly:
\begin{equation}
\label{eq: Ivan IV}
C^{4}=
\frac{1-\gamma}{1+\gamma}
\left(\frac{\sigma+\gamma}{\sigma-\gamma}\right)^\sigma\,,
\hskip 40pt\text{if }\
p\in(1,\infty),
\end{equation}
where
$$
\sigma=\cos\phi_p=|1-2/p|
\hskip 40pt
\text{and}
\hskip 40pt
\gamma=
\frac{\sqrt{\sigma^2-\cos^2\phi}}{|\sin\phi|}\,,
$$
and 
\begin{equation*}
\label{massena}
C(\phi,1)=C(\phi,\infty)=\lim_{p^*\rightarrow\infty}C(\phi,p)=\frac1{\sqrt{\cos\phi}}\,.
\end{equation*}
\end{thm}

An analogous theorem by Epperson \cite{Epp} concerns the $L^{p}$ norm of the semigroup $\exp(-z\Delta_{OU})_{z\in\C_+}$ on $L^{p}(\R^n,\mu)$, where $1<p<\infty$, $\Delta_{OU}$ is the $n$-dimensional positive {\it Ornstein-Uhlenbeck operator} and $\mu$ the standard Gau\ss ian measure on $\R^n$. Epperson's findings are dichotomous: in the region $\mn{z\in\overline{\C_+}}{|\sin\Im z|\leq (\tan\phi_p)\sinh\Re z}$ the semigroup is contractive on $L^{p}(\R^n,\mu)$, while elsewhere it is not even bounded. His result immediately implies that $\overline\bS_{\phi_p}$ is the largest {\sl sector} in which $\exp(-z\Delta_{OU})_{z\in\C_+}$ is bounded (or contractive) on $L^{p}(\R^n,\mu)$, as in the non-Gau\ss ian case.

The evaluation of the $L^{p}$-norms of $\exp(z\Delta_n)$, for all $z\in\C_+$, 
is due to Weissler \cite[Theorem~3 (b)]{We} in the case $p^*\geq3$. 
His formulation differs from \eqref{eq: Ivan IV}, though of course they are equivalent. 
The result for $p^*\in(2,3)$ was confirmed a decade later, by combining 
\cite{Epp} with the remark made in \cite{We} just after Theorem~3.

Prior to learning about the paper by Weissler \cite{We}, we proved Theorem~\ref{t: tgv bordeaux-cdg} 
as follows. First we used the so-called Beckner's tensorization trick \cite[Lemma~2]{Be} to reduce the calculation of the norm to the case $n=1$. Then we applied a result of Epperson \cite[Theorem~2.4]{Epp}, see also Lieb \cite[Theorem~4.1]{L}, according to which in order to calculate the norm in $L^{p}(\R)$ of $\exp(z\,d^2/dx^2)$ with $\Re z>0$ it suffices to test the operator on {\it centered Gau\ss ian functions}.

\begin{rem}
The expression \eqref{eq: Ivan IV} 
implicitly also appeared in \cite{DPV}, in a completely different context. Indeed, the calculation on p. 508 there 
(with $t=1/\sigma$ and $w=\gamma/\sigma$) directly confirms that $C>1$ when $0<\gamma<\sigma<1$. 
\end{rem}

 \begin{rem}
 We see that, for $p\ne2$, $\lim_{|\phi|\rightarrow\pi/2} C(\phi,p) =\infty$.
 This reflects the fact that $e^{is\Delta_n}$ is unbounded on $L^{p}(\R^n)$ for any $s\in\R$ and 
 $p\in[1,\infty]\backslash\{2\}$, see \cite{H,AEMH}.
 \end{rem}

\begin{proof}[Proof of Proposition~\ref{p: reki begalci}.]
\label{begalci 2015}
Let $p,\vartheta,A_n$ be as in the formulation of the proposition. Then the estimate \eqref{eq: besko}, applied with $A=A_n$ and $B=A_n^*$, and Theorem~\ref{t: tgv bordeaux-cdg} imply that
\[
N_p(A_n,A_n^*)
\geq
\frac1{2}\Nor{e^{-L_{A_n}}}_{\cB(L^{p}(\R^n))}
=
\frac1{2}\, C(
\vartheta,p)^n
\rightarrow\infty
\hskip 25pt
\text{ as }n\rightarrow\infty.
\qedhere
\]
\end{proof}

\section{Proof of the contractivity result (Theorem~\ref{t: Aufnahmebetriebsartenwaehler})}
\label{masamune}

In this section we give a proof of Theorem~\ref{t: Aufnahmebetriebsartenwaehler}. It involves no heat flow and no Bellman function. It does, however, prominently feature {\it power} functions and their convexity. Furthermore, we elucidate the connection between this paper and the ones by Nittka \cite{N} and Cialdea and Maz'ya \cite{CM}. The reader may also consult the latter authors' subsequent monograph \cite{CM2014} and a paper by Cialdea \cite{C}.

We use a characterization of the contractivity on $L^{p}(\Omega)$ of the semigroup $(P^{A}_{t})_{t>0}$ which is due to Nittka \cite{N} 
and relies on earlier results by Ouhabaz \cite[Theorem~2.2]{O}; see the discussion and references in \cite{N}. Unlike the classical Lumer--Phillips theorem which characterizes it in terms of $L_{A}$ itself, the result below does it in terms of the quadratic form associated with $L_{A}$. Its main convenience for our purpose is that the domains of $L_{A}$ and especially their realizations on $L^{p}$ are not known, while the domain of the quadratic form in \eqref{eq: miruna} is by definition the Sobolev space $H_0^1(\Omega)$. See also \cite[p. 43]{O}.

\subsection{$L^{p}$-dissipativity of forms}
The notion of {\it $L^{p}$-dissipativity} of sesquilinear forms was introduced by Cialdea and Maz'ya in \cite[Definition~1]{CM} for the case of forms defined on $C_c^1(\Omega)$ and associated with complex matrices. Motivated by the desire to merge \cite{CM} and Nittka \cite[Theorem~4.1]{N}, we extend that notion as follows.

\begin{defin}
\label{linna}
Let $X$ be a measure space, $\gotb$ a sesquilinear form defined on the domain $\cD(\gotb)\subset L^2(X)$ and $1<p<\infty$. Denote 
$\cD_p(\gotb):=\mn{u\in\cD(\gotb)}{|u|^{p-2}u\in\cD(\gotb)}$. We say that $\gotb$ is {\it $L^{p}$-dissipative} if 
$$
\Re\gotb\left(u,|u|^{p-2}u\right)\geq0
\hskip40pt
\forall\,u\in\cD_p(\gotb).
$$
\end{defin}

Recall \cite[Definition 1.5]{O} that the {\it adjoint form} of $\gotb$ is defined by $\cD(\gotb^*):=\cD(\gotb)$ and $\gotb^*(u,v):=\overline{\gotb(v,u)}$. The following is a straightforward characterization of $L^p$-dissipativity. 

\begin{prop}
\label{p: dissipativity dual}
Let $X,\gotb,p$ be as in Definition \ref{linna} and $q$ given by $1/p+1/q=1$. The following statements are equivalent:
\begin{itemize}
\item
$\gotb$ is $L^{p}$-dissipative;
\item
$\gotb^*$ is $L^{q}$-dissipative;
\item
$\Re\gotb\left(|v|^{q-2}v,v\right)\geq0$ for all $v\in\cD_q(\gotb)$.
\end{itemize}
\end{prop}
From Proposition~\ref{p: dissipativity dual} it is obvious that Definition \ref{linna} indeed extends \cite[Definition~1]{CM}.

\begin{thm}[Nittka]
\label{t: dzigerica}
Suppose that $\Omega\subset\R^n$ is open, $A\in\cA(\Omega)$ and $p\in(1,2)$. Then $(P^{A}_{t})_{t>0}$ extends to a contractive operator semigroup on $L^{p}(\Omega)$ if and only if the form \eqref{eq: miruna} is $L^{p}$-dissipative.
\end{thm}

\begin{proof}
Apply the equivalence (i)$\Leftrightarrow$(iii) from \cite[Theorem~4.1]{N} to the form \eqref{eq: miruna}. One only needs to make sure that the orthogonal projection $L^2(\Omega)\rightarrow\mn{u\in L^2\cap L^{p}}{\nor{u}_p\leq1}$ preserves $H_0^1(\Omega)$. 

By taking $A\equiv I_n$, the form \eqref{eq: miruna} gives rise to the ordinary Laplacian subject to the Dirichlet boundary conditions on $\Omega$. The associated semigroup is contractive on $L^{p}(\Omega)$ for all $p\in[2,\infty]$, see \cite[Theorem~4.7]{O}. Thus, by applying \cite[Theorem~4.1]{N}, we conclude that the orthogonal projection of $L^2(\Omega)$ onto the set $\mn{u\in L^2\cap L^{p}}{\nor{u}_p\leq1}$ preserves $H_0^1(\Omega)$. For $1<p<2$ use duality.
\end{proof}

The $L^p$-dissipativity of the form \eqref{eq: miruna} studied, as noted above, in \cite{CM} and \cite{N} is closely related to another integral condition which recently appeared in ter \label{sensordust} Elst et al. \cite[(1.13) and Theorem 1.7]{tELSV}.
Their condition on $A=U+iV$ requires that, for all $v\in H_0^1(\Omega)$,
\begin{equation}
\label{eq: liveviewfocus}
\int_\Omega
\left(
\Re\sk{A\nabla v}{\nabla v}
-\widehat{p}^2\sk{U\nabla|v|}{\nabla|v|}
-2|\widehat{p}|\,|\!\sk{V_{\mathsf s}\nabla|v|}{\Im(\text{sign}\,{\bar v}\nabla v)}\!|
\right)
\geq0.
\end{equation}
If $p>2$ and in the third term of \eqref{eq: liveviewfocus} one removes the absolute value signs except on $|v|$, then for $u=|v|^{2/p-1}v$ one gets precisely $\Re\text{\got a}\left(u,|u|^{p-2}u\right)$, where $\text{\got a}$ is as in \eqref{eq: miruna}.

\begin{prop}
\label{p: that's the way}
Let $\text{\got a}$ be as in \eqref{eq: miruna} and $p>1$. If $\text{\got a}$ is $L^{p}$-dissipative then $\Delta_p(A_{\mathsf s})\geq0$.
\end{prop}

\begin{proof}
It is enough to assume that $p\geq2$. The case $1<p<2$ can then be obtained by duality, see Proposition \ref{p: dissipativity dual}, Corollary \ref{c: garrison} {\it \ref{eq: haeri}.)}  and Proposition \ref{p: kaminszki}.

So take $p\geq2$. The proposition follows by examining Cialdea and Maz'ya \cite{CM}.
If $\text{\got a}$ is $L^{p}$-dissipative then, see \cite[(2.18)]{CM}, 
$$
\int_\Omega\sk{\cB_\mu\nabla r}{\nabla r}_{\R^n}\geq0
$$
for all real $r\in C_c^1(\Omega)$ and all $\mu\in\R$, where, writing $A=U+iV$ and recalling \eqref{eq: the conspirators},
$$
\cB_\mu:=(1+\mu^2)U+2\mu\oV_q(V).
$$
As in \cite[p. 1076]{CM} we conclude that, if $\alpha\in\R^n$ is arbitrary, $\sk{\cB_\mu(x)\alpha}{\alpha}dm(x)$ is a positive measure on $\Omega$, therefore for all $\alpha\in\R^n$ and $\mu\in\R$ we have
$$
\sk{\cB_\mu(x)\alpha}{\alpha}\geq0 
\hskip 30pt
\text{a.e. }x\in\Omega.
$$

Write 
$
\cS_\mu=\cS_\mu(\alpha):=\mn{x\in\Omega}{\sk{\cB_\mu(x)\alpha}{\alpha}\geq0}.
$
We just saw that $m(\Omega\backslash\cS_\mu)=0$ for any $\mu\in\R$. Define also
$
\cS=\cS(\alpha):=\mn{x\in\Omega}{|\sk{\oV_q(V)(x)\alpha}{\alpha}|\leq\sk{U(x)\alpha}{\alpha}}.
$
Clearly,
$$
\cS=\bigcap_{\mu\in\R}\cS_\mu=\bigcap_{\mu\in\Q}\cS_\mu,
$$
therefore
$
m(\Omega\backslash\cS)=0.
$
We proved that, for every $\alpha\in\R^n$,
\begin{equation}
\label{eq: cas solo}
|\sk{\oV_q(V)(x)\alpha}{\alpha}|\leq\sk{U(x)\alpha}{\alpha}
\hskip 30pt
\text{a.e. }x\in\Omega.
\end{equation}
Observe that, again by \eqref{eq: the conspirators}, 
$$
 \sk{\oV_q(V)(x)\alpha}{\alpha}
=\frac{p-2}{2\sqrt{p-1}}\sk{V(x)\alpha}{\alpha}.
$$
So \eqref{eq: cas solo} means that, for every $\alpha\in\R^n$,
$$
|p-2||\sk{V(x)\alpha}{\alpha}|\leq2\sqrt{p-1}\sk{U(x)\alpha}{\alpha}
\hskip 30pt
\text{a.e. }x\in\Omega,
$$
which by Proposition~\ref{p: trazi se morricone} implies that $\Delta_p(A_{\mathsf s})\geq0$.
\end{proof}

Recall that the condition $\div\, W^{(k)}=0$ was introduced on page \pageref{s: svoju rabotu ceram} and that it was earlier studied in \cite{ABBO} and \cite{O}.

\begin{lem}
\label{l: El Hob}
Let $W\in L^{\infty}(\Omega\rightarrow\R^{n,n})$ be antisymmetric. Assume that, for every $k\in\{1,\hdots,n\}$, we have $\div\, W^{(k)}=0$ in the distributional sense. Then 
$$
\int_\Omega\sk{W\nabla f}{\nabla g}_{\C^n}=0
\hskip 40pt
\forall\, f,g\in H_0^1(\Omega)\,.
$$
\end{lem}
\begin{proof}
Clearly it suffices to assume that $f,g\in C_c^\infty(\Omega)$. Write $W=[w_{jk}]_{j,k}$. 
Let $\Lambda_w$ denote the distribution corresponding to the function $w\in L^{\infty}(\Omega)$. 
We recollect the basic properties of calculus with distributions \cite[Chapter 6]{R}. Then 
$$
\int_\Omega\sk{W\nabla f}{\nabla g}_{\C^n}
=\sum_{j,k}\Lambda_{w_{jk}}(\partial_{x_k}f\cdot\pd_{x_j}\bar g)
=-\sum_k\sum_j(\pd_{x_j}\Lambda_{w_{jk}})(\partial_{x_k}f\cdot\bar g)=0\,.
$$
Antisymmetry of $W$ was used for the middle equality.
\end{proof}

We remind the reader of the notation
$$
\cD_p(\gota)=\Mn{u\in H_0^1(\Omega)}{|u|^{p-2}u\in H_0^1(\Omega)}.
$$
The $L^{p}$-dissipativity of the form \eqref{eq: miruna} is closely related to our $H_{F_p}^A$, as we show next.

\begin{prop}
\label{p: kao na strazi cuvam noc od budnih}
Suppose that $A\in
L^{\infty}(\Omega\rightarrow\C^{n,n})$, $p\geq2$ and $f\in\cD_p(\gota)$. Then
\begin{equation}
\label{eq: da li ti se svidjam}
p\,\Re\sk{A\nabla f}{\nabla(|f|^{p-2}f)}_{\C^n}
=H_{F_p}^A[f;\nabla f]\,.
\end{equation}
When $\div(\Im A)_{\sf a}^{(k)}=0$ for all $k\in\{1,\hdots,n\}$  
then also
\begin{equation}
\label{eq: div=0}
p\,\Re\int_{
\Omega}\sk{A\nabla f}{\nabla(|f|^{p-2}f)}_{\C^n}
=\int_{\Omega} H_{F_p}^{A_{\sf s}}[f;\nabla f]\,.
\end{equation}
\end{prop}

\begin{proof}[Sketch of the proof.]
Define $G(\zeta):=|\zeta|^{p-2} \zeta$ for $\zeta\in\C$. Fix a cutoff function $\eta\in C^{\infty}_{c}(\C)$ such that $\eta\equiv1$ on $\{|\zeta|\leq 1\}$ and $\eta\equiv0$ on $\{|\zeta|>2\}$, and consider $[\eta(\cdot/R) G]\circ f$. To this function apply a version of the chain rule for weak derivatives \cite[Theorem~2.1.11]{Z} adapted to complex functions and  take the limit as $R\rightarrow\infty$. This leads to the identity 
\begin{equation}
\label{eq: osta san inkantan}
\nabla\left(|f|^{p-2}f\right)=\frac p2|f|^{p-2}\left(\nabla f+\widehat p\,e^{2i\arg f}\nabla\bar f\right).
\end{equation}
In order to arrive at \eqref{eq: da li ti se svidjam} it now suffices to recall Lemma~\ref{l: tatac}.

Now we prove the last statement of the proposition. By \eqref{eq: nikkor 24mm 2.8} and \eqref{eq: ichimonji},
$$
H_{F_p}^A=H_{F_p}^{A_{\sf s}}+H_{F_p}^{i(\Im A)_{\sf a}}.
$$
By the part of the proposition proven so far,
\begin{equation}
\label{eq: zurba}
\aligned
\int_{\Omega} 
H_{F_p}^{i(\Im A)_{\sf a}}[f;\nabla f] 
& = -p\,\Im\int_{\Omega}\sk{(\Im A)_{\sf a}\nabla f}{\nabla(|f|^{p-2}f)}_{\C^n}.
\endaligned
\end{equation}
The last integral turns out to be purely imaginary. Lemma~\ref{l: El Hob} implies that it is zero.
\end{proof}

\begin{rem}
Previous versions of Proposition~\ref{p: kao na strazi cuvam noc od budnih} featured simple algebraic conditions which permitted \eqref{eq: div=0}, yet were stronger than the zero-divergence condition. A question by Michael Cowling and a discussion with Gian Maria Dall'Ara following that question eventually led us to replace earlier algebraic conditions in favour of the condition on divergence. We thank both of our colleagues.
\end{rem}

\subsection{Proof of Theorem~\ref{t: Aufnahmebetriebsartenwaehler}}
\label{s: jozic}

It suffices to treat the case $p\geq2$. Indeed, contractivity of $P^{A}_{t}$ on $L^{p}(\Omega)$ is equivalent to the contractivity of $P^{A^{*}}_{t}$ on $L^{q}(\Omega)$, where $1/p+1/q=1$; on the other hand, nonnegativity of $\Delta_p(A)$ or $\Delta_p(A_{\mathbf s})$ is preserved under changing either $p\leftrightarrow q$ or $A\leftrightarrow A^*$ (see Proposition~\ref{p: kaminszki} and Corollary~\ref{c: garrison}).

\medskip

{\it(\ref{Imshi})} $\Rightarrow$ {\it(\ref{Wara})}: 
Follows from Theorem~\ref{t: dzigerica}, Proposition~\ref{p: kao na strazi cuvam noc od budnih} and Proposition~\ref{p: kaminszki}.

\medskip
{\it(\ref{eq: Kidbuhom})} $\Rightarrow$ {\it(\ref{Wara})} 
{\it if $\div(\Im A)_{\sf a}^{(k)}=0$ for all $k\in\{1,\hdots,n\}$}:
The same.

\medskip
{\it(\ref{Wara})} 
$\Rightarrow$ {\it(\ref{eq: Kidbuhom})}: 
Combine Theorem~\ref{t: dzigerica} with Proposition~\ref{p: that's the way}.

\medskip
{\it(\ref{eq: Kidbuhom})} $\not\Rightarrow$ {\it(\ref{Wara})} {\it if $\div(\Im A)_{\sf a}^{(k)}\ne0$ for some $k\in\{1,\hdots,n\}$}:
The simplest case of such $A\in\cA(\Omega)$ occurs when $n=2$. 
In that case $(\Im A)_{\sf a}=wR$ for some $w\in L^{\infty}(\Omega)$ and 
\begin{equation*}
\label{eq: nessun dorma}
R=\left[
\begin{array}{lr}
& -1\\
1 &
\end{array}
\right],
\end{equation*}
and nonzero divergence is equivalent to saying that $w$ is nonconstant on each connected component of $\Omega$.

So take $\Omega=\R^2$, a nonconstant $w\in L^{\infty}(\R^2)$ and define
$A(x)=I_{\R^2}+iw(x)R$ for $x\in\R^2$. 
By Lemma~\ref{l: adventni koledar 2018}, $A\in\cA(\R^2)$ if and only if $\nor{w}_{\infty}<1$. Since $A_{\mathbf s}=I_2$ is real, we have $\Delta_p(A_{\mathbf s})>0$, so {\it(\ref{eq: Kidbuhom})} is fulfilled.

Take $p>2$, real functions $r,\f\in C^\infty(\R^2)$ such that $r\geq0$ and $f:=re^{i\f}$ belongs to $\cD_p(\gota)$. Then
\begin{equation}
\label{eq: 7.4}
\Re\sk{A\nabla f}{\nabla(|f|^{p-2}f)}_{\C^2}
=
(p-1) r^{p-2}|\nabla r|^2
+
 r^p|\nabla\f|^2
+w\,
\oJ(r^p,\f).
\end{equation}
Here $\oJ(F,G)$ is the Jacobian determinant of a pair of real (weakly) differentiable functions $F,G$ on $\R^2$:
$$
\oJ(F,G):=
\bigg|
\begin{array}{ll}
\pd_xF & \pd_yF \\
\pd_xG & \pd_yG 
\end{array}
\bigg|
=
\pd_xF\cdot \pd_yG - \pd_yF\cdot \pd_xG
=\sk{R\nabla F}{\nabla G}.
$$ 
We want to find $w,r,\f$ such that \eqref{eq: 7.4} will be negative after integration. We first choose $w(x)=-\gamma\chi_E(x)$, where $\gamma\in(0,1)$ and $E\subset\R^2$ is a Borel set with $m(E)m(E^c)>0$. We are trying to minimize the integral of \eqref{eq: 7.4}, which suggests arranging $r,\f,E$ so that
\begin{equation}
\label{eq: disapproval}
E:={\rm supp\,}\oJ(r^p,\f)_+={\rm supp\,}\oJ(r,\f)_+,
\end{equation} 
where $F_+:=\max\left\{F,0\right\}$ is the positive part of the real function $F$.

Suppose furthermore that we can write $r=e^{\sigma/p}$ for some real function $\sigma\in C^\infty(\R^2)$. 
With \eqref{eq: disapproval} in mind we obtain
\begin{equation}
\label{eq: no absolution}
\Re\int_{\R^2}\sk{A\nabla f}{\nabla(|f|^{p-2}f)}_{\C^2}
=
\int_{\R^2} 
e^{\sigma}\left(
\frac
{|\nabla\sigma|^2}{pq}
+
|\nabla\f|^2
-
\gamma\oJ(\sigma,\f)_+
\right).
\end{equation}

A quick argument shows that for {\it any} pair $\sigma,\f$ for which the set $E$ from \eqref{eq: disapproval} has positive measure, our problem has a solution for large $p$. Indeed, replace $\f$ in \eqref{eq: no absolution} by $\mu\f$ with positive $\mu$, consider $\mu\searrow0$ and finally choose large $p$.

This pattern leads to the following concrete example: 
\begin{itemize}
\item
$E=\mn{(x_1,x_2)\in\R^2}{|x_1|\geq|x_2|}$
\item
$0<\gamma<1$
\item
$A(x)=I_{\R^2}-i\gamma\chi_E(x)R$ for $x\in\R^2$
\item
${\displaystyle f(x_1,x_2)=e^{-\pi(x_1^2+x_2^2)-ipx_1x_2}}$.
\end{itemize}
Then we verify that for $p>2\pi(\pi+\sqrt{\pi^2-1})\approx 38.45$ there exists $\gamma$ close to 1 such that 
$\Re\int_{\R^2}\sk{A\nabla f}{\nabla(|f|^{p-2}f)}_{\C^2}<0$. Hence {\it(\ref{Wara})} is not fulfilled, by Theorem~\ref{t: dzigerica}.

Finally observe that, since $\chi_E$ is a pointwise limit of $C^\infty$ functions with values in $[0,1]$, we see from \eqref{eq: 7.4} that we have counterexamples from $\cA(\R^n)$ with $C^\infty$ entries.
\qed

\begin{rem}
Matrices of the form $A(x)=I_{\R^2}+iw(x)R$ were used earlier by Cialdea and Maz'ya, see \cite[Examples 1 and 2]{CM}, however to give counterexamples to phenomena different from the one just considered in the proof of Theorem~\ref{t: Aufnahmebetriebsartenwaehler}. 
\end{rem}

\subsection*{Discussion}
As seen from the last step of the proof of Proposition~\ref{p: kao na strazi cuvam noc od budnih}, what really matters for obtaining \eqref{eq: div=0} is to find a (preferably algebraic) condition on $(\Im A)_{\sf a}$ and $p$ under which   
\begin{equation}
\label{eq: blag}
\int_{\Omega}\sk{(\Im A)_{\sf a}\nabla f}{\nabla(|f|^{p-2}f)}_{\C^n}=0
\hskip40pt
\forall f\in\cD_p(\gota)\,.
\end{equation}
This makes us wonder whether the characterization {\it(\ref{Wara})} $\Leftrightarrow$ {\it(\ref{eq: Kidbuhom})} from Theorem~\ref{t: Aufnahmebetriebsartenwaehler} holds {\sl only} if \eqref{eq: blag} holds. The counterexample which establishes the last part of Theorem~\ref{t: Aufnahmebetriebsartenwaehler} was in fact chosen as the simplest one for which \eqref{eq: blag} {\sl fails}. Moreover, we wonder whether in the abscence of \eqref{eq: blag} we might have {\it(\ref{Wara})} $\Leftrightarrow$ {\it(\ref{Imshi})}. 
In light of this, let us further discuss the setup that led to the example at the end of the proof of Theorem~\ref{t: Aufnahmebetriebsartenwaehler}. 

Let $\Omega\subset\R^2$ be open and $E\subset\Omega$ such that $m(E)m(\Omega\backslash E)>0$. Suppose that the functions $r,\f\in C^1(\Omega)$ are real and such that $r\geq0$ and $f:=re^{i\f}\in\cD_p(\gota)$. As before, let $\gamma\in(0,1)$ and  $A(x):=I_{\R^2}-i\gamma\chi_E(x)R$. Further mimic the previous example by assuming that $r=e^{\sigma/p}$ for some real $\sigma\in C^1(\Omega)$ and $\psi:=\sqrt{pq}\f$. Then we get
\begin{equation}
\label{eq: paspalj}
pq\,\Re\int_{\Omega}\sk{A\nabla f}{\nabla(|f|^{p-2}f)}
=
\int e^{\sigma}
\left(
|\nabla\sigma|^2
+|\nabla\psi|^2
-\gamma\sqrt{pq}\,
\oJ(\sigma,\psi)\chi_E\right).
\end{equation}
It seems that the threshold for uniform positivity of \eqref{eq: paspalj} is $\gamma\sqrt{pq}=2$. Indeed, write 
$\gamma\sqrt{pq}=2+\delta$ for some $\delta\in\R$. We quickly see that, for $\delta\leq0$, the integrals above are always nonnegative. On the other hand, the right-hand side of \eqref{eq: paspalj} is
$$
\int_Ee^\sigma\left[(\sigma_x-\psi_y)^2+(\sigma_y-\psi_x)^2\right]
-\delta
\int_Ee^\sigma\left(\sigma_x\psi_y-\sigma_y\psi_x\right)
+\int_{E^c}e^\sigma\left(\sigma_x^2+\sigma_y^2+\psi_x^2+\psi_y^2\right),
$$
so if $\delta>0$ is sufficiently large then \eqref{eq: paspalj} may be negative. It is natural to ask whether it could be negative for {\it any} $\delta>0$ (and a suitable choice of $E,\sigma,\psi$). This may be viewed as a variational problem. It seems to us that the extremizers could be approximate solutions to the Cauchy--Riemann system. Observe also that, since $m(E)>0$, Lemma~\ref{l: adventni koledar 2018} gives $\Delta_p(A)=1-\sqrt{\wh p^2+\gamma^2}$, therefore $\gamma\sqrt{pq}\leq2$ (or $\delta\leq0$) is nothing but $\Delta_p(A)\geq0$.

Simplifying by taking $\Omega=\R^2$ and functions from the Schwartz class $\cS(\R^2)$ instead of $\cD_p(\gota)$, and then writing $\rho=e^{\sigma/2}$ and $\omega=\psi/2$, we reformulate the question as follows:

Find $E\subset\R^2$ with $m(E)m(E^c)>0$, such that for any $\delta>0$ there exist functions $\rho>0$, $\omega$ real, for which $\rho e^{i\omega}\in\cS(\R^2)$ and
$$
\int_{\R^2}\left|\nabla(\rho e^{i\omega})\right|^2<(1+\delta)\int_E\oJ(\rho^2,\omega).
$$
This would give an example of an accretive matrix for which the $L^{p}$ contractivity of the semigroup is for any $p>1$ characterized by $\Delta_p(A)\geq0$ and not $\Delta_p(A_{\mathsf s})\geq0$.

\section*{Appendix: 
Approximation with regular-coefficient operators}
\label{s: sargija}

\renewcommand{\thesection}{\Alph{section}}
\setcounter{section}{1}
\setcounter{thm}{0}

For any $\phi\in(0,\pi/2)$ denote 
$
\phi^*:=\pi/2-\phi.
$ 
Write $I=[0,1]$. Consider a one-parameter family  $\mn{A_s}{s\in I}$ of $n\times n$ complex matrix functions. 
Recall the notation \eqref{eq: Don Giovanni}. Suppose that there exist $\lambda,\Lambda>0$ such that
\begin{enumerate}[(H1)]

\item 
\label{I: H1}
$A_s\in\cA_{\lambda,\Lambda}(\R^n)$ for all $s\in I$;

\item 
\label{I: H2}
${\displaystyle\lim_{s\rightarrow 0}\nor{A_s(x)-A_0(x)}_{\cB(\C^n)}=0}$  for a.e. $x\in\R^{n}$.
\end{enumerate}
Denote by $L_{s}=L_{A_{s}}$ the divergence-form operator on $L^{2}(\R^{n})$ associated with $A_{s}$. Recall from Section~\ref{s: njuska} that each $L_{s}$ is sectorial with sectoriality angle
$
\omega(L_{s})\leq\theta_{\lambda,\Lambda},
$
where $\theta_{\lambda,\Lambda}=\arccos(\lambda/\Lambda)$.
Set $\cH_0=L^2(\R^n)$ and $\cH_1=L^2(\R^n;\C^n)$. 
For $\zeta\in\C\setminus \bS_{\theta_{\lambda,\Lambda}}$ define the operators
$T_s(\zeta):
\cH_0\rightarrow\cH_1$ 
and
$S_s(\zeta):C^{\infty}_{c}(\R^{n};\C^{n})\rightarrow \cH_{0}$ by 
$$
\aligned
T_{L_{s}}(\zeta)&:=\nabla(\zeta-L_s)^{-1}\\
\quad S_{L_{s}}(\zeta)&:=(\zeta-L_s)^{-1}{\rm div}.
\endaligned
$$

\begin{lem}[see \text{ \cite[Section~1.2, Proposition~1]{AT}}]
\label{l: sect app}
Assume (H1) holds. Then for every $\theta\in(\theta_{\lambda,\Lambda},\pi/2)$ there exists 
$C=
C(\lambda,\Lambda,\theta)>0$ 
such that
$$
\aligned
|\zeta|^{1/2}\|
 T_{L_{s}}(\zeta)f\|_{\cH_1}&\leq C\|f\|_{\cH_0};\\
 |\zeta|^{1/2}\|
 S_{L_{s}}(\zeta)F\|_{\cH_0}&\leq C\|F\|_{\cH_1};\\
 \|\nabla
S_{L_{s}}(\zeta)F\|_{\cH_1}&\leq C\|F\|_{\cH_1};
\endaligned
$$ 
for all $s\in I$, $f\in L^2(\R^{n})$, $F\in C^{\infty}_{c}(\R^{n};\C^{n})$ and $\zeta\in \C\setminus\bS_{\vartheta}$. 
Moreover, the very same estimates hold with $L_{s}$ replaced by $L^{*}_{s}$.
\end{lem}

\begin{proof}
Let us start with the first inequality.
The condition (H1) implies that
$$
\aligned
\lambda\|\nabla(\zeta-L_{s})^{-1}f\|_{\cH_1}^{2}
&\leq\Re\!\sk{ A_s\nabla (\zeta-L_{s})^{-1}f}{\nabla(\zeta-L_{s})^{-1}f}_{\cH_1}\\
&= \Re\!\sk{L_{s}(\zeta-L_{s})^{-1}f}{(\zeta-L_{s})^{-1}f}_{\cH_0}\\
&= \Re\!\sk{\zeta(\zeta-L_{s})^{-1}f}{(\zeta-L_{s})^{-1}f}_{\cH_0}-\Re\!\sk{f}{(\zeta-L_{s})^{-1}f}_{\cH_0}.
\endaligned
$$
The first inequality of the lemma now follows from the sectoriality of $L_{s}$. Since $L^{*}_{s}=L_{A^{*}_{s}}$ and $A^{*}_{s}\in\cA_{\lambda,\Lambda}(\R^n)$, the same estimate clearly holds with $L_{s}$ replaced by $L_{s}^*$. The second inequality now follows from duality by using the identity
$$
\sk{S_{L_{s}}(\zeta)F}{f}_{\cH_{0}}=\sk{F}{T_{L^{*}_{s}}(\bar\zeta)f}_{\cH_{1}},\quad F\in C^{\infty}_{c}(\R^{n};\C^{n}),\quad f\in \cH_{0}.
$$

We now prove the third inequality. See \cite[Section~0.2, Proposition~2]{AT}. By the accretivity condition of $L_{s}$ we have, similarly as before,
$$
\aligned
\lambda&\|\nabla(\zeta-L_{s})^{-1}{\rm div}F\|^{2}_{\cH_1}
\\
&\hskip 20pt
\leq\Re\sk{\zeta(\zeta-L_{s})^{-1}{\rm div}F}{(\zeta-L_{s})^{-1}{\rm div}F}_{\cH_0}-\Re\sk{{\rm div}F}{(\zeta-L_{s})^{-1}{\rm div}F}_{\cH_0}\\
&\hskip 20pt
=(\Re
\zeta)\nor{S_{L_{s}}(\zeta)F}_{\cH_0}^2
-\Re\sk{F}{\nabla S_{L_{s}}(\zeta)F}_{\cH_1}\\
&\hskip 20pt
\leq |\zeta|\cdot\nor{S_{L_{s}}(\zeta)F}^{2}_{\cH_0}+\nor{F}_{\cH_1}\nor{\nabla S_{L_{s}}(\zeta)F}_{\cH_1}.
\endaligned
$$
The second estimate of this lemma -- which we have already proven -- gives
$$
\lambda\,\frac{\nor{\nabla S_s(\zeta)F}^2_{\cH_1}}{\nor{F}^2_{\cH_1}}
\leq \frac{C}\lambda+\frac{\nor{\nabla S_s(\zeta)F}_{\cH_1}}{\nor{F}_{\cH_1}}\,.
$$
This quickly gives 
$\lambda\nor{\nabla S_s(\zeta)F}_{\cH_1}/\nor{F}_{\cH_1}\leq1+C$ 
for all $s\in I$ and $\zeta\in \C\setminus\bS_{\vartheta}$.
\end{proof}

\begin{rem}
The preceding lemma implies that for $\zeta\in\C\setminus\bS_{\theta_{\lambda,\Lambda}}$ the operators $S_{L_{s}}(\zeta)$ and $\nabla S_{L_{s}}(\zeta)$ admit unique extensions to  bounded operators $\cH_1\rightarrow\cH_0$ and $\cH_{1}\rightarrow\cH_{1}$, respectively. Moreover,  $S^{*}_{L_{s}}(\zeta)=T_{L^{*}_{s}}(\bar \zeta)$. 
\end{rem}

The following auxiliary result is mentioned in \cite[p. 56]{AT}. We omit the proof.

\begin{lem}
\label{l: pert resolvent}
For every $s\in I$ , $\zeta\in\C\setminus \bS_{\vartheta^{*}_{2}}$ and $f\in L^{2}(\R^{n})$ we have
$$
\left(\zeta-L_0\right)^{-1}f-\left(\zeta-L_{s}\right)^{-1}f=S_s(\zeta)\circ M_{A_0-A_s}\circ T_0(\zeta)f,
$$
where $M_{A_0-A_s}:\cH_{1}\rightarrow\cH_{1}$ denotes the operator of multiplication by $A_0-A_s$.
\end{lem}

\begin{lem}
\label{c: appregop}
Fix $f\in L^{2}(\R^{n})$.
Assuming (H1) and (H2), for every $z\in\bS_{\theta^{*}_{\lambda,\Lambda}}$ we have
$$
\nabla e^{-zL_{s}}f\rightarrow \nabla e^{-zL_0}f\,
$$
in $L^{2}(\R^{n};\C^{n})$, as $s\rightarrow 0$.
\end{lem}

\begin{proof}
Let $\theta^{*}\in(0,\pi/2)$ be such that $|\arg z|<\theta^{*}<\theta^{*}_{\lambda,\Lambda}$.
Fix $\delta>0$ and denote by $\gamma$ the positively oriented 
boundary of $\bS_{\vartheta}\cup\mn{\zeta\in\C}{|\zeta|<\delta}$.
For $s\in I$ and $\zeta\in \gamma$, define
$$
U(s,\zeta)=S_{L_{s}}(\zeta)\circ M_{A_0-A_s}\circ T_{L_{0}}(\zeta).
$$
Then by \cite[Lemma~2.3.2]{Haase} and Lemma~\ref{l: pert resolvent},
$$
\nabla e^{-zL_0}f-\nabla e^{-zL_{s}}f=\frac{1}{2\pi i}\int_{\gamma}e^{-z\zeta}\nabla U(s,\zeta)f\,d\zeta.
$$
Therefore, by Lemma~\ref{l: sect app}, 
$$
\aligned
\norm{\nabla e^{-zL_0}f-\nabla e^{-zL_{s}}f}{2}&
\leqsim 
\int_{\gamma}|e^{-z\zeta}|\cdot\norm{\nabla U(s,\zeta)f}{2}\,d |\zeta|\\
&\leqsim
\int_{\gamma}e^{-\Re(z\zeta)}\norm{M_{A_0-A_s}T_0(\zeta)f}{2}\,d |\zeta|\,.
\endaligned
$$
By (H1), (H2), Lemma~\ref{l: sect app} and the Lebesgue dominated convergence theorem 
we obtain
$$
\lim_{s\rightarrow 0}\norm{M_{A_0-A_s}T_0(\zeta)f}{2}=0,\quad \forall \zeta\in\gamma.
$$ 
Moreover, by Lemma~\ref{l: sect app} again, we have
$
\norm{M_{A_0-A_s}T_0(\zeta)f}{2}\leqsim|\zeta|^{-1/2}\norm{f}{2}.
$
The desired conclusion 
now follows from the Lebesgue dominated convergence theorem. 
\end{proof}

\setcounter{subsection}{0}

\subsection{Convolution with approximate identity}
\label{s: BWV 1056 7. Allegro moderato}
Here we give an example of when (H\ref{I: H1}) and (H\ref{I: H2}) may be fulfilled. 
 
Let $k:\R^n\rightarrow[0,1]$ be a radial, non negative, compactly supported $C^\infty$ function whose integral over $\R^n$ equals one. For $\e>0$ define 
$k_\e(x):=\e^{-n}k(x/\e)$.
If $A\in\cA(\R^n)$ we define $A_\e:=A*k_\e$, meaning that $A_\e$ is a matrix function whose entries are $a_{ij}*k_\e$. Recall the notation \eqref{eq: Don Giovanni} and \eqref{eq: kabuto}.

\begin{lem}
\label{l: stevens}
For every $A\in\cA_{\lambda,\Lambda}(\R^n)$ and $A_\e$ as above and $p>1$ we have:
\begin{enumerate}[ i)]
\item
\label{l: igor tukaj celi dan u racnalniku gasi}
${\displaystyle\lim_{\varepsilon\rightarrow0}\nor{A_\e(x)-A(x)}_{\cB(\C^n)}=0}$ for a.e. $x\in\R^n$;
\item
\label{l: jz bom zeu lestve}
$A_\e\in\cA_{\lambda,\Lambda}(\R^n)$ for all $\e>0$;
\item
\label{l: na to bom splezau, ves}
$\nu(A_\e)\leq\nu(A)$ for all $\e>0$;
\item
\label{l: A5 item iii}
$\Delta_{p}(A_{\varepsilon})\geq \Delta_{p}(A)$
and 
${\displaystyle\Delta_{p}(A)=\lim_{\varepsilon\rightarrow 0}\Delta_{p}(A_{\varepsilon})}$;
\item
\label{severus}
$\mu(A_\e)\geq\mu(A)$
and
${\displaystyle\mu(A)=\lim_{\e\rightarrow0}\mu(A_\e)}$.
\end{enumerate}
\end{lem}

\begin{proof}
Since pointwise convergence in $\cB(\C^{n})$ implies convergence in norm, 
for proving item {\it \ref{l: igor tukaj celi dan u racnalniku gasi})} it is enough to show that for a.e. $x\in\R^n$ we have $A_{\varepsilon}(x)\xi\rightarrow A(x)\xi$ for all $\xi\in\C^{n}$. This is true because each $a_{ij}$ belongs to $L^{\infty}(\R^{n})\subset L^{1}_{{\rm loc}}(\R^{n})$ so that $(a_{ij}*k_\e)(x)$  tends to $a_{ij}(x)$ for a.e. $x\in\R^n$, as $\e\rightarrow 0$. This proves {\it \ref{l: igor tukaj celi dan u racnalniku gasi})}.

Clearly, for all $\xi,\eta\in\C^n$,
$$
\sk{A_\e(x)\xi}{\eta}=\int_{\R^n}\sk{A(y)\xi}{\eta}k_\e(x-y)\,dy\,,
$$
which directly yields {\it\ref{l: jz bom zeu lestve})}, {\it\ref{l: na to bom splezau, ves})} and the inequality in {\it\ref{l: A5 item iii})}. By definition \eqref{eq: kabuto} and the continuity of $x\mapsto A_{\varepsilon}(x)$ we get
$\Re\sk{A_{\varepsilon}(x)\xi}{\cI_{p}\xi}\geq \Delta_{p}(A_{\varepsilon})|\xi|^{2}$ for all $x\in\R^{d},\xi\in\C^{d}$ and $\varepsilon>0$. 
The equality in {\it\ref{l: A5 item iii})} now follows from item {\it\ref{l: igor tukaj celi dan u racnalniku gasi})}.

Finally, let us prove {\it\ref{severus})}. 
We will repeatedly be applying Proposition~\ref{p: reci dje si}. 
Choose any $p>1$ such that $|1-2/p|\leq\mu(A)$. This implies $\Delta_p(A)\geq0$, thus by {\it\ref{l: A5 item iii})} also $\Delta_p(A_\e)\geq0$ for any $\e>0$, hence $|1-2/p|\leq\mu(A_\e)$. We know that $\mu(A)\leq1$ always, therefore $|1-2/p|$ can get arbitrarily close to $\mu(A)$. This proves the inequality in {\it\ref{severus})}.
Now we prove the equality. If $\mu(A)=1$, the part just proven gives $\mu(A_\e)=1$, so this case is settled. If $\mu(A)<1$, let $p>1$ be such that $\mu(A)=|1-2/p|$. Part {\it\ref{l: A5 item iii})} then gives $\Delta_p(A_\e)\searrow0$ as $\e\rightarrow0$. Since Proposition~\ref{p: reci dje si} and part {\it\ref{l: jz bom zeu lestve})} imply 
$$
\frac{\Delta_p(A_\e)}{\Lambda}
\leq
\mu(A_\e)-|1-2/p|
\leq
\frac{\Delta_p(A_\e)}{\lambda},
$$
this completes the proof of the lemma. 
\end{proof}

\bigskip
\footnotesize
\noindent\textit{Acknowledgments.}
The first author was partially supported by the ``National Group
for Mathematical Analysis, Probability and their Applications'' (GNAMPA-INdAM).
The second author was partially supported by the Ministry of Higher Education, Science and Technology of Slovenia (research program Analysis and Geometry, contract no. P1-0291). He is deeply grateful to Pascal Auscher for their conversations during the 2011 Oberwolfach conference ``Real Analysis, Harmonic Analysis and Applications'' and to El Maati Ouhabaz for several discussions which took place in Bordeaux and Novo mesto in 2014 and 2015, respectively. 

The authors would like to thank the referees of this paper for their valuable comments which improved its presentation.

\end{document}